\documentclass[a4paper]{article}

\usepackage[utf8]{inputenc}
\usepackage[T1]{fontenc}
\usepackage[francais]{babel}
\usepackage{amsthm}
\usepackage{varioref}
\usepackage{amssymb,amsmath}
\usepackage{lmodern}
\usepackage[a4paper]{geometry}
\usepackage{dsfont}

\title{Dimension topologique, moyenne dimension et théorèmes de plongements}
\author{Fanny Amyot\\
Mémoire de Master 2,
préparé sous la direction de Yonatan Gutman, \\
co-dirigé par David Burguet
}

\begin{document}
\newtheorem{Def}{Définition}
\labelformat{Def}{la définition~#1}
\newtheorem{Lemme}{Lemme}
\labelformat{Lemme}{le lemme~#1}
\newtheorem{propo}{Proposition}
\labelformat{propo}{la proposition~#1}
\newtheorem{thm}{Théorème}
\labelformat{thm}{le théorème~#1}
\newtheorem{coro}{Corollaire}
\labelformat{coro}{le corollaire~#1}
\theoremstyle{remark}
\newtheorem{Rem}{Remarque}
\newtheorem{Ex}{Exemple}
\newtheorem{Notation}{Notation}

\maketitle

\begin{abstract}
La rédaction de ce mémoire est basée principalement sur la lecture des deux articles \cite{LW} et \cite{YG}. Il s'agissait d'appréhender la moyenne dimension topologique et de comprendre l'intérêt de cette notion, notamment dans l'étude de théorèmes de plongements. J'ai dû me documenter préalablement sur la dimension topologique, et j'y consacre donc la première partie du mémoire. Dans la seconde, on introduit la notion de moyenne dimension topologique et dans la dernière, on étudie plusieurs cas particuliers de la conjecture de Lindenstrauss-Tsukamoto.

Je remercie Yonatan Gutman pour cette introduction à ce sujet et à la recherche en général, et pour sa disponibilité malgré la distance. Je remercie aussi David Burguet pour les discussions que nous avons eues et pour avoir co-dirigé mon master, et Frédéric Le Roux pour bien avoir voulu faire partie du jury lors de la soutenance.
\end{abstract}

\tableofcontents

\section{Dimension topologique}

Cette section est motivée par deux objets : l'introduction de la dimension périodique dans la troisième section, et l'étude de la moyenne dimension topologique. Celle-ci est définine comme une moyenne de la dimension topologique et on remarquera qu'elle se comporte relativement de la même façon. On utilise d'ailleurs les mêmes méthodes et outils dans les démonstrations des résultats de la section 2.

Dans ce chapitre, X est un espace topologique non vide.

\subsection{Définition}

\begin{Def}

Soit $\alpha$ un recouvrement de X. Soit $x\in X$. On pose :
\[ ord_x(\alpha)=\sum_{U\in \alpha} (\mathds{1}_U(x)) -1\]
et :
\[ ord(\alpha)=\sup_{x\in X} ord_x(\alpha)\]
On dit que $ord(\alpha)$ est l'ordre de $\alpha$.
\end{Def}

On a que $ord(\alpha) \in \mathbb{N} \cup \infty$.  Dans la suite, on considèrera des recouvrements ouverts finis, on aura donc $ord(\alpha) \in \mathbb{N}$.

Soient $\alpha=(A_i)_{i\in I}$ et $\beta=(B_j)_{j\in J}$ des recouvrements de X. On dit que $\beta$ est un raffinement de $\alpha$ si pour tout $j\in J$ il existe $i \in I$ tel que $B_j \subset A_i$. On note $\beta \succ \alpha$.

\begin{Def}
Soit $\alpha$ un recouvrement ouvert fini de X. On pose :
\[ D(\alpha)=\min_\beta ord(\beta) \]
où $\beta$ parcourt tous les recouvrements ouverts finis de X tels que $\beta \succ \alpha$.
\end{Def}

\begin{Rem}

1. On a $D(\alpha)\in \mathbb{N}$

2. $D(\alpha)\leq n$ si et seulement si il existe $\beta$ un recouvrement ouvert fini raffinant $\alpha$ tel que $ord(\beta)\leq n$.

3. Si $\alpha$ et $\tilde{\alpha}$ sont deux recouvrements ouverts finis tels que $\alpha \succ \tilde{\alpha}$, alors $D(\alpha)\geq D(\tilde{\alpha})$.
\end{Rem}

\begin{Def}
La dimension topologique de X est définie par :
\[ dim(X) = \sup_\alpha D(\alpha) \]
où $\alpha$ parcourt tous les recouvrements ouverts finis de X.
\end{Def}

\subsection{Premières propriétés}

\begin{propo}\label{dimfermé}
Soit $F\subset X$ un sous-ensemble fermé de X. Alors $dim(F)\leq dim(X)$.
\end{propo}

\begin{proof}
Soit $\alpha$ un recouvrement ouvert fini de F. Il existe une famille d'ouverts de X, $({U_i})_{i\in I}$, telle que : $F\subset \bigcup_{i\in I}U_i$. Comme $X\backslash F$ est un ouvert de X, $\beta =({U_i})_{i\in I}\cup (X\backslash F)$ est un recouvrement ouvert fini de X. Soit $\gamma$ un recouvrement ouvert fini de X raffinant $\beta$ tel que $ord(\gamma)=D(\beta)$. Soit $\tilde{\gamma}=(U\cap F)_{U\in \gamma}$. Alors $\tilde{\gamma}$ est un recouvrement ouvert fini de F raffinant $\alpha$, et tel que $ord(\tilde{\gamma})\leq ord(\gamma)$
Donc $ord(\alpha)\leq ord(\tilde{\gamma})\leq ord(\gamma)=D(\beta)$. Finalement, $dim(F)\leq dim(X)$.
\end{proof}

\begin{Lemme}
Soit $\alpha=(U_i)_{i\in I}$ un recouvrement ouvert fini de X. Alors on peut trouver un recouvrement $\beta=(V_i)_{i\in I}$ tel que pour tout $i\in I$, $V_i\subset U_i$ et $ord(\beta)\leq D(\alpha)$.
\end{Lemme}

\begin{proof}
Soit $\gamma=(W_j)_{j\in J}$ un recouvrement ouvert fini de X tel que $\gamma \succ \alpha$ et $ord(\gamma)=D(\alpha)$. Alors on peut définir une application $\phi :J\to I$ telle que pour tout $j\in J$, $W_j\subset U_{\phi (j)}$. Alors posons :
\[ \beta=(V_i)_{i\in I}=(\bigcup_{j\in \phi^{-1}(i)} W_j)_{i\in I} \]
$\beta$ est un recouvrement ouvert fini de X et pour tout $i\in I$, on a $V_i\subset U_i$. De plus, soit $x\in X$. Alors $x\in V_i$ si et seulement s'il existe j dans $\phi^{-1}(i)$ tel que $x \in W_j$. On en déduit que $ord_x(\beta)\leq ord_x(\gamma)$. Donc finalement, ceci étant vrai pour tout $x\in X$, $ord(\beta)\leq D(\alpha)$.
\end{proof}

\begin{Lemme}\label{préunion}
Soit F un sous-ensemble fermé de X. Soit $\alpha=(U_i)_{i\in I}$ un recouvrement ouvert fini de X. Alors on peut trouver un recouvrement $\beta=(V_i)_{i\in I}$ tel que pour tout $i\in I$, $V_i\subset U_i$ et $ord_x(\beta)\leq dim(F)$ pour tout $x\in F$.
\end{Lemme}

\begin{proof}
Considérons le recouvrement ouvert fini de F : $(F\cap U_i)_{i\in I}$. Alors d'après le lemme précédent, on peut trouver un recouvrement $\gamma=(W_i)_{i\in I}$ tel que pour tout $i\in I$, $W_i\subset F \cap U_i$ et $ord(\gamma)\leq dim(F)$. Pour tout $W_i$ il existe $O_i$ un ouvert de X tel que $W_i=F\cap O_i$. Posons :
\[ \beta=(V_i)_{i\in I}= ((O_i\cup(X\backslash F))\cap U_i)_{i\in I}\]
Alors $\beta$ est un recouvrement ouvert fini de X tel que pour tout $i\in I$, $V_i \subset U_i$. De plus, si $x\in F$, $x\in V_i$ si et seulement si $x\in W_i$. Donc $ord_x(\beta)=ord_x(\gamma)\leq ord(\gamma)\leq dim(F)$.
\end{proof}

\begin{propo}\label{union}
Soient F et G deux sous-ensembles fermés de X. Alors :
\[ dim(F\cup G)=max(dim(F),dim(G))\]
\end{propo}

\begin{proof}
D'après la proposition précédente, on sait déjà que $dim(F\cup G)\geq max(dim(F),dim(G))$.

Soit $\alpha$ un recouvrement ouvert fini de $F\cup G$. En appliquant \ref{préunion} à $\alpha$ et à F, on obtient un recouvrement $\beta=(U_i)_{i\in I}$ de $F\cup G$. Puis on obtient un recouvrement $\gamma=(V_i)_{i\in I}$ de $F\cup G$ en appliquant de nouveau \ref{préunion} à $\beta$ et à G. $\gamma$ est tel que pour tout $i\in I$, $V_i \subset U_i$. Alors pour $x\in G$, on a $ord_x(\gamma)\leq dim(G)$. D'autre part, si $x\in V_i \in \gamma$, alors $x\in U_i$. Donc $ord_x(\gamma)\leq ord_x(\beta)\leq dim(F)$. Finalement, $ord_x(\gamma)\leq max(dim(F),dim(G))$. Comme $\gamma \succ \alpha$, on en déduit que $D(\alpha)\leq max(dim(F),dim(G))$. On a choisi $\alpha$ arbitrairement donc on peut conclure : $dim(F\cup G)\leq max(dim(F),dim(G))$.
\end{proof}

\begin{Def}
Soit $\alpha$ une famille finie de sous-ensembles de X. On définit :
\[maille(\alpha)=\max_{U \in \alpha} diam(U)\]
\end{Def}

\begin{propo}\label{maille}
Supposons que X est un espace métrique compact. Soit $n\in \mathbb{N}$. Alors $dim(X)\leq n$ si et seulement si pour tout $\epsilon>0$, il existe un recouvrement ouvert fini de X, $\alpha$ tel que $maille(\alpha)\leq \epsilon$ et $ord(\alpha)\leq n $.
\end{propo}
\begin{proof}
Supposons $dim(X)\leq n$. Soit $\epsilon>0$. Alors par compacité de X, on peut trouver $\alpha$ un recouvrement fini formé de boules ouvertes de rayon $\frac{\epsilon}{2}$. Alors $D(\alpha)\leq n$, donc il existe $\beta \succ \alpha$ tel que $ord(\beta)\leq n$. On a bien $maille(\beta)\leq \epsilon$.

Réciproquement, soit $\alpha$ un recouvrement ouvert fini de X. Pour tout $U\in \alpha$ et tout $x\in U$, il existe une boule ouverte centrée en x, de rayon $r_x$, inclue dans U. Les boules ouvertes de la forme $B(x,\frac{r_x}{2})$ recouvrent X, et par compacité de X, on peut en extraire un recouvrement fini $(B(x,\frac{r_x}{2}))_{x\in A}$. Alors posons $\eta=\min_{x\in X}\frac{r_x}{2}$. Soit $\beta$ tel que $maille(\beta)\leq \eta$ et $ord(\beta)\leq n$. Soit $V\in \beta$. Alors prenons $y_V$ un élément de V. Il existe $x\in A$ tel que $y_V\in B(x,\frac{r_x}{2})$ et $B(x,r_x)\subset V$. Alors pour tout $y\in V$,  $d(y,x)\leq diam(V) +d(y_V,x)<r_x$. Comme $B(x,r_x)$ est inclue dans un ouvert de $\alpha$, on en déduit que Y est inclus dans un ouvert de $\alpha$. Finalement, $\beta \succ \alpha$. Donc $D(\alpha)\leq n$. Donc finalement, $dim(X)\leq n$.
\end{proof}

\subsection{Dimension topologique des polyèdres}

Les polyèdres vont être des outils particulièrement importants dans l'étude de la dimension topologique. En particulier, on montre ici que la dimension topologique étend à des espaces topologiques la notion de dimension algébrique d'un espace vectoriel.

\subsubsection{Polyèdres}

\begin{Def}
Soit $\{ v_1, v_2,..., v_r\}$ une famille de points de $\mathbb{R}^m$. Alors on dit qu'elle est affinement indépendante si $\{v_2-v_1,v_3-v_1,...,v_{r}-v_1\}$ est une famille libre.
\end{Def}

\begin{Def}
Soit $r\in \mathbb{N}$. On dit que $\Delta$ est un r-simplexe de $\mathbb{R}^m$ si c'est l'enveloppe convexe d'un ensemble de r+1 points de $\mathbb{R}^m$, $\{ p_0, p_1,..., p_r\}$, affinement indépendants. Ces points sont appelés les sommets de $\Delta$. Une face de $\Delta$ est un simplexe dont les sommets sont pris parmi les sommets de $\Delta$.
\end{Def}

\begin{Def}
Un complexe simplicial de $\mathbb{R}^n$ est un ensemble fini $C$ de simplexes de $\mathbb{R}^n$ vérifiant :
\begin{enumerate}
\item si $\Delta \in C$, et si $F$ est une face de $\Delta$, alors $F\in C$.
\item si $\Delta \in C$ et $\Delta '\in C$, alors $\Delta \cap \Delta '$ est une face commune à $\Delta$ et $\Delta '$.
\end{enumerate}
\end{Def}

Les sommets de $C$ sont les sommets des simplexes composant $C$.

On appelle support de $C$, et on note $|C|$, le sous-ensemble de $\mathbb{R}^n$ défini par l'union de tous les simplexes composant $C$.

On appelle dimension combinatoire de $C$ la dimension maximale des simplexes composant $C$.

\begin{Def}
Un espace topologique X est un polyèdre s'il existe un complexe simplicial $C$ tel que X soit homéomorphe au support $|C|$ de $C$.
\end{Def}

\begin{Def}
Un complexe simplicial abstrait est la donnée d'un couple $(V, \Sigma)$, où $V$ est un ensemble fini dont les éléments sont appelés les sommets et $\Sigma$ un ensemble de parties de V, appelées simplexes, vérifiant qu'une sous-partie d'un simplexe est aussi un simplexe.
\end{Def}

Soit $(V, \Sigma)$ un complexe simplicial abstrait. On va lui associer un complexe simplicial $C$ qu'on appellera sa réalisation géométrique.
Posons $n=card(V)$ et $V={v_1,v_2,...,v_n}$. On note $(e_1,e_2,...,e_n)$ la base canonique de $\mathbb{R}^n$. Alors à tout simplexe $\sigma$ on associe le simplexe de $\mathbb{R}^n$ ayant pour ensemble de sommets : $\{ e_i|v_i\in \sigma \}$. Alors l'ensemble de ces simplexes est bien un complexe simplicial de $\mathbb{R}^n$.

Réciproquement, soit $C$ un complexe simplicial de $\mathbb{R}^n$. Alors si V est l'ensemble des sommets de $C$ et $\Sigma$ est composé des parties de V qui sont les ensembles des sommets d'un simplexe de $C$, le couple $(V,\Sigma)$ est un complexe simplicial abstrait. On dit que c'est le complexe simplicial abstrait associé à $C$.

\begin{Ex}
Soit $V=\{ 1,2,3,4\}$ et $\Sigma =\{ \emptyset, \{ 1\} ,\{ 2\} ,\{ 3\} ,\{ 4\} ,\{ 2,3\} ,\{ 3,4\} ,\{ 2,4\} ,\{ 1,2\} ,\{ 2,3,4\} \}$ C'est un complexe simplicial abstrait. Sa réalisation géométrique est un complexe simplicial de $\mathbb{R}^4$, composé de 4 1-simplexes, 4 2-simplexes, et un 3-simplexe.
\end{Ex}

Voici un autre exemple important dans la suite :

\begin{Def}\label{nerf}
Soit $\alpha =(A_i)_{i\in \mathbb{N}}$ un recouvrement d'un ensemble X. Alors on appelle nerf de $\alpha$ le complexe simplicial abstrait $(V,\Sigma)$ où $V=I$ et $\Sigma$ est l'ensemble des $J\subset I$ tels que
\[ \bigcap_{j\in J}A_j\neq \emptyset \]
\end{Def}

Dans la fin de cette section, on commence à relier la dimension combinatoire d'un polyèdre et sa dimension topologique. On verra dans la suite qu'elles sont en fait égales. C'est un résultat important pour la théorie de la dimension topologique, notamment grâce à \ref{reci}.

\begin{Def}
Soit s un sommet d'un complexe simplicial $C$. La réunion de tous les simplexes de $C$ ne contenant pas s est un ensemble fermé. On note $E_C(s)$ son complémentaire dans $|C|$. On dit que c'est l'étoile ouverte de s.
\end{Def}

\begin{propo}\label{étoile}
Soit $C$ un complexe simplicial de $\mathbb{R}^n$, et S l'ensemble de ses sommets. Alors $(E_C(s))_{s\in S}$ est un recouvrement ouvert fini de $|C|$ d'ordre m, où m est la dimension combinatoire de $C$.
\end{propo}

\begin{proof}
Soit $x\in |C|$. Alors il existe $\Delta$ un simplexe de $C$ tel que $x\in \mathring{\Delta}$. Soit s un sommet de $\Delta$. Alors $x\in E_C(s)$. En effet, soit $\tilde{\Delta}$ un simplexe contenant x. Quitte à prendre une face de $\tilde{\Delta}$, on peut supposer que $x\in \mathring{\tilde{\Delta}}$. Alors $\Delta \cap \tilde{\Delta}$ est une face de $\Delta$ qui rencontre $\mathring{\Delta}$, donc c'est $\Delta$. De même on montre $\Delta \cap \tilde{\Delta}=\tilde{\Delta}$ donc $\Delta =\tilde{\Delta}$. Donc $s\in \tilde{\Delta}$. Finalement, $x\in E_C(s)$. On en déduit que $(E_C(s))_{s\in S}$ est un recouvrement ouvert fini de $|C|$.

Soit maintenant $x\in |C|$, et $ s_1,...,s_r$  les sommets tels que $x\in E_C(s_i)$. Soit $\Delta$ tel que $x\in \Delta$. Alors $ s_1,...,s_r$ sont des sommets de $\Delta$. Donc r est strictement inférieur à la dimension combinatoire de $\Delta$, donc à m. Ceci étant vrai pour tout x, $ord(\alpha)\leq m$.

Réciproquement, soit $\Delta$ un m-simplexe de $C$, de sommets $ s_0,...,s_m$, et $x\in \mathring{\Delta}$. Alors $x\in E_C(s)$ pour tout $s\in \{ s_0,...,s_m\}$. On en déduit $ord(\alpha)\geq m$. Finalement on a bien $ord(\alpha)= m$.
\end{proof}

On admet ici un lemme (technique), démontré dans \cite{Coo}.
\begin{Lemme}
Soit $C$ un complexe simplicial de $\mathbb{R}^n$. Alors pour tout $N\in \mathbb{N}$, on peut construire un complexe simplicial $C_N$, qui a même dimension combinatoire et même support que $|C|$, et tel que :
\[ \lim_{N\to \infty}\max_{\Delta \in C_N}diam(\Delta)=0 \]

\end{Lemme}

\begin{propo}\label{combi}
Soit $|C|$ un complexe simplicial de $\mathbb{R}^n$, et m sa dimension combinatoire. Alors $dim(|C|)\leq m$.
\end{propo}

\begin{proof}
Soit $N\in \mathbb{N}$. Soit $\alpha_N=(E_{C_N}(s))_{s\in S_N}$, où $S_N$ est l'ensemble des sommets de $C_N$, défini dans le lemme précédent. Alors d'après \ref{étoile}, $\alpha_N$ est un recouvrement ouvert fini de $|C_N|$, donc de $|C|$, d'ordre m.

Alors $maille(\alpha_N)\leq 2(\max_{\Delta \in C_N}diam(\Delta))$. En effet, si $s\in S$ et $x,y\in E_C(s)$, il existe $\Delta_x$ et $\Delta_y$ contenant respectivement x et y et ayant s pour sommet. Alors on a :
\[ d(x,y)\leq d(x,s)+d(s,y)\leq diam(\Delta_1)+diam(\Delta_2)\]
Donc $diam(E_{C_N}(s))\leq 2(\max_{\Delta \in C_N}diam(\Delta))$, d'où le résultat.

On en déduit que $\lim_{N\to \infty} maille(\alpha_N) =0$. D'après \ref{maille}, on a donc que $dim(|C|)\leq m$.

\end{proof}

\subsubsection{Lemme de Lebesgue}

Le lemme de Lebesgue va nous permettre en particulier d'achever le calcul de la dimension  des polyèdres.

\begin{propo}\label{Lebesgue}
Soit $n\in \mathbb{N}$ et $\alpha$ un recouvrement ouvert fini du cube $[0,1]^n$, tel qu'aucun élément de $\alpha$ ne rencontre deux faces opposées de $[0,1]^n$. Alors on a $D(\alpha)\geq n$.
\end{propo}

\begin{proof}
Soit $\beta=(U_i)_{i\in I}$ un recouvrement ouvert fini raffinant $\alpha$. Montrons que $ord(\beta)\geq n$. Supposons le contraire : $ord(\beta)<n$. On va construire une fonction $\psi : [0,1]^n \to [0,1]^n$ continue qui n'admet pas de point fixe. On aura alors obtenu une contradiction avec le théorème du point fixe de Brouwer.

Pour tout $i\in I$, on considère le sommet $S(i)\in \{0,1\}^n$, défini par :

\[ S(i)|_k=
\left\{
\begin{array}{ll}
1 \text{ si }U_i \text{ rencontre le bord } \{ x_k=0 \} \\
0\text{ sinon}
\end{array}
\right.
\]

Alors si $U_i$ rencontre une face, $S(i)$ se situe sur la face opposée. En effet, comme $\beta \succ \alpha$, aucun élément de $\beta$ n'intersecte deux faces opposées du cube  $[0,1]^n$.
Soit $(\rho_i)_{i\in I}$ une partition de l'unité adaptée au recouvrement $\beta$. On considère la fonction $\phi$ définie par :
\[ \forall x\in [0,1]^n, \phi(x)=\sum_{i\in I}\rho_i(x)S(i)\]
Alors si x est sur une face de $[0,1]^n$, $\phi (x)$ appartient à la face opposée.

D'autre part, soit $x\in [0,1]^n$. Alors $\phi(x)$ est dans l'enveloppe convexe de la famille $\{ S(i) | x\in U_i \}$, qui est de cardinal < n car $ord(\beta)<n$. Donc l'image de $\phi$, notée $Im(\phi)$, est inclue dans une réunion finie de sous-espaces affines de $[0,1]^n$ de dimension $<n$. Ces espaces sont fermés et d'intérieur vide dans $[0,1]^n$, donc d'après le théorème de Baire, cette union est d'intérieur vide. Donc $Im(\phi)$ est d'intérieur vide dans $[0,1]^n$.

Il existe donc un élément $\omega$ dans $[0,1]^n\backslash Im(\phi)$. On considère $\pi :[0,1]^n\backslash \{ \omega \} \to \partial [0,1]^n$ la projection qui à un point x associe l'intersection du bord de $[0,1]^n$ avec la demi-droite partant de $\omega$ et passant par x. En particulier, $\pi$ laisse fixe tout point du bord.

Considérons l'application $\psi = \pi \circ \phi : [0,1]^n \to [0,1]^n$. C'est bien une application continue, qui n'admet pas de point fixe. En effet, son image est inclue dans le bord de $[0,1]^n$. Or tout point du bord est envoyé par $\psi$ sur la face opposée du cube.

\end{proof}

\begin{coro}
Soit $n\in \mathbb{N}$. Alors $dim([0,1]^n)=n$.
\end{coro}

\begin{proof}
$[0,1]^n$ est homéomorphe au support d'un n-simplexe de $\mathbb{R}^n$. D'après \ref{combi}, on a donc $dim([0,1]^n)\leq n$. D'autre part, soit $\alpha$ un recouvrement ouvert fini de $[0,1]^n$, tel qu'aucun ouvert dans $\alpha$ ne rencontre deux faces opposées du cube $[0,1]^n$. Alors d'après \ref{Lebesgue}, $D(\alpha)\geq n$. Donc $dim([0,1]^n)\geq n$.
\end{proof}

On obtient ainsi des exemples d'espaces de dimension n pour tout entier n. On a aussi $dim([0,1]^\mathbb{N})=\infty$. En effet, pour tout $n\in \mathbb{N}$, $[0,1]^n$ est un sous-espace fermé de $[0,1]^\mathbb{N}$, donc $n=dim[0,1]^n)\leq dim([0,1]^\mathbb{N})$.

\begin{coro}\label{dimpolyedre}
Soit $n\in \mathbb{N}$ et $C$ un complexe simplicial de $\mathbb{R}^n$. Alors la dimension topologique du support $|C|$ de $C$ est la dimension combinatoire de $C$.
\end{coro}

\begin{proof}
On sait déjà d'après \ref{combi} que $dim(|C|)\leq m$. Soit $\Delta$ un m-simplexe de $C$. Alors il est homéomorphe à $[0,1]^m$, donc de dimension m. Donc d'après \ref{dimfermé}, $dim(|C|)\geq m$.
\end{proof}

\begin{coro}
Soient P et Q des polyèdres. Alors :
\[ dim(P\times Q)=dim(P)+dim(Q)\]
\end{coro}

\begin{proof}
On va montrer que si $C$ et $C'$ sont deux complexes simpliciaux, $dim(|C|\times |C'|)=dim(|C|)+dim(|C'|)$. Soient $\Delta$ un k-simplexe de $C$ et $\Delta '$ un k'-simplexe de $C'$. Alors $\Delta \times \Delta '$ est homéomorphe à $[0,1]^k\times [0,1]^{k'}$, donc est de dimension k+k'. Or : \[ |C|\times |C'|=\bigcup_{\Delta \in C, \Delta '\in C'}\Delta \times \Delta ' \]
Il s'agit d'une union finie d'ensembles fermés, donc d'après \ref{union}, $dim(|C|\times |C'|)=\max_{\Delta \in C, \Delta '\in C'}dim(\Delta \times \Delta ')$. Comme la dimension d'un complexe simplicial est égale à sa dimension combinatoire, $dim(|C|\times |C'|)=dim(|C|)+dim(|C'|)$.
\end{proof}

\subsubsection{Applications $\alpha$-compatibles}

\begin{Def}
Soit une application $f:X\to Y$. Soit $\alpha$ un recouvrement ouvert fini de X. On dit que $f$ est $\alpha$-compatible s'il existe un recouvrement ouvert fini $\beta$ de Y, tel que le recouvrement $f^{-1}(\beta )$ raffine $\alpha$.
\end{Def}

\begin{propo}\label{compat}
Supposons que X est compact. Soit $\alpha$ un recouvrement ouvert fini de X.  Soit une application continue $f:X\to Y$ telle que pour tout $y\in Y$, il existe $U\in \alpha$ tel que $f^{-1}(y)\subset U$. Alors f est $\alpha$-compatible.
\end{propo}

\begin{proof}
Si $\alpha=(A_i)_{i\in I}$, posons : $\beta=(B_i)_{i\in I}=(\{y\in Y | f^{-1}(y)\subset A_i \})_{i\in I}$. Alors $\beta$ est un recouvrement fini de Y tel que $f^{-1}(\beta ) \succ\alpha$. Il reste à montrer que les $B_i$ sont ouverts. Si un ensemble $B_i$ n'est pas ouvert, alors il existe $y\in B_i$ et une suite $(y_n)_{n\in \mathbb{N}}$ d'éléments de $Y\backslash B_i$, convergeant vers y. On a alors :
\[ \forall n\in \mathbb{N}, \exists x_n \in  f^{-1}(y_n)\cap (X\backslash A_i) \]
Comme X est compact, la suite $(x_n)_{n\in \mathbb{N}}$ admet une sous-suite convergeant vers un élément $x\in X$. Par continuité de $f$, $f(x)=y$, donc $x\in A_i$. Mais d'autre part, comme $X\backslash A_i$ est fermé, $x\in X\backslash A_i$. C'est une contradiction, donc finalement les $B_i$ sont des ouverts de Y.
\end{proof}

\begin{propo}\label{reciproque}
Soit une application $f:X\to Y$. Si $\beta$ est un recouvrement ouvert fini de X, alors $D(f^{-1}(\beta))\leq D(\beta)$.
\end{propo}

\begin{proof}
Soit $\gamma$ un recouvrement ouvert fini tel que $\gamma \succ \beta$ et $ord(\gamma)=D(\beta)$. Alors on a $ord(f^{-1}(\gamma))\succ ord(\gamma)$. En effet, si $A\in \gamma$ et $x\in f^{-1}(A)$ alors $f(x)\in A$ donc :
\[ \sum_{A\in \alpha} (\mathds{1}_{f^{-1}(A)}(x)) -1\leq \sum_{A\in \alpha} (\mathds{1}_A(f(x)) -1\]
Finalement, comme $f^{-1}(\gamma)\succ f^{-1}(\beta)$, on a $D(f^{-1}(\beta))\leq D(\beta)$.
\end{proof}

\begin{propo}
Soit $\alpha$ un recouvrement ouvert fini de X et $f:X\to Y$ une application $\alpha$-compatible. Alors $D(\alpha)\leq dim(Y)$.
\end{propo}

\begin{proof}
$f$ est $\alpha$-compatible donc il existe $\beta$ un recouvrement ouvert fini de Y tel que $f^{-1}(\beta)\succ \alpha$. Alors d'après la proposition précédente, $D(\alpha)\leq D(f^{-1}(\beta))\leq D(\beta)\leq dim(Y)$.
\end{proof}

\begin{propo}\label{reci}
Supposons que X soit un espace métrique compact, et $\alpha$ un recouvrement ouvert fini de X. Alors il existe un espace Y de dimension topologique $D(\alpha)$ et une application $f:X\to Y$ continue et $\alpha$-compatible.
\end{propo}

On peut construire l'espace Y comme un polyèdre.

\begin{proof}
Soit $\beta=(B_i)_{i\in I}$ un recouvrement ouvert fini de X qui raffine $\alpha$ et tel que $ord(\beta)=D(\alpha)$. Soit $C$ la réalisation géométrique du nerf de $\beta$ et Y le polyèdre $|C|$. On garde les mêmes notations que dans \ref{nerf}. Soit $(\rho_i)_{i\in I}$ une partition de l'unité subordonnée à $\beta$. Alors on définit :
\begin{align*}
f:\: &X\longrightarrow Y&\\
&x\longmapsto \sum_{i\in I} \rho_i(x) e_i&
\end{align*}
$f$ est bien continue, et $dim(Y)$ est la dimension combinatoire de $C$ d'après \ref{dimpolyedre}, donc c'est $ord(\beta)$, c'est-à-dire $D(\alpha)$.

Enfin, f est $\beta$-compatible d'après \ref{compat}. En effet, soit $y\in Y$ et $\Delta \subset I$ le simplexe de $C$ contenant y de dimension minimale (il est unique). Alors si $i\in \Delta$, et $f(x)=y$, $\rho_i(x)\neq 0$ donc $x\in B_i$. Donc $f^{-1}(y)\subset B_i$. Finalement, comme $\beta \succ \alpha$, f est $\alpha$-compatible.
\end{proof}

\begin{coro}\label{produit}
Soient X et Y deux compacts métrisables. Alors :
\[ dim(X\times Y)\leq dim(X)+dim(Y)\]
\end{coro}

\begin{proof}
Soit $\gamma$ un recouvrement ouvert fini de $X\times Y$. Alors par définition de la topologie produit sur $X\times Y$, chaque élément de $\gamma$ peut s'écrire comme une réunion (finie car $X\times Y$ est compact) d'ouverts élémentaires $U\times V$. Donc on peut trouver deux recouvrements ouverts finis $\alpha$ et $\beta$ respectivement de X et de Y, tels que $(\alpha \times \beta)=(U\times V)_{U\in \alpha, V\in \beta} $ raffine $\gamma$.

D'après \ref{reci}, on peut trouver P un polyèdre de dimension $D(\alpha)$, Q un polyèdre de dimension $D(\beta)$, et deux applications continues, $f:X\to P$, $\alpha$-compatible, et $g:Y\to Q$, $\beta$-compatible. Posons :
\begin{align*}
h:&X\times Y \longrightarrow P\times Q&\\
&(x,y) \longmapsto (f(x),g(y))&
\end{align*}

Alors h est $(\alpha \times \beta)$-compatible donc $\gamma$-compatible. Donc $D(\gamma)\leq dim(P\times Q)$.

\begin{Def}
Soient $\alpha=(A_i)_{i\in I}$ et $\beta=(B_j)_{j\in J}$ des recouvrements de X. On note $\alpha \vee \beta $ et on appelle joint de $\alpha$ et $\beta$ le recouvrement : $(A_i \cap B_j)_{(i\in I, j\in J)}$.
\end{Def}

\begin{Rem}
On a $\alpha \vee \beta \succ \alpha$ et $\alpha \vee \beta \succ \beta$
\end{Rem}

\begin{propo}\label{joint}
Soient $\alpha$ et $\beta$ des recouvrements de X. Alors :
\[ D(\alpha \vee \beta)\leq D(\alpha)+D(\beta) \]
\end{propo}

\begin{proof}
D'après \ref{reci}, on peut trouver Y un polyèdre de dimension $D(\alpha)$, Y' un polyèdre de dimension $D(\beta)$, et deux applications continues, $f:X\to Y$, $\alpha$-compatible, et $g:X\to Y'$, $\beta$-compatible. Posons :
\begin{align*}
h:&X \longrightarrow Y\times Y'&\\
&x \longmapsto ((f(x),g(x))&
\end{align*}
Alors h est $(\alpha \vee \beta)$-compatible. Donc $D(\alpha \vee \beta)\leq dim ( Y\times Y')$. Comme Y et Y' sont compacts, on obtient :
\[D(\alpha \vee \beta)\leq dim(Y)+dim(Y')= D(\alpha)+D(\beta)\]
\end{proof}

\end{proof}

\subsection{Plongements topologiques}

Soient X et Y deux espaces topologiques. On dit que X se plonge topologiquement dans Y s'il existe une application $f:X\to Y$ qui induit un homéomorphisme de X sur $f(X)$. $f$ est appelée un plongement de X dans Y.

La dimension topologique étant un invariant topologique, d'après \ref{dimfermé}, si X se plonge topologiquement dans Y, on a $dim(X)\leq dim(Y)$.

Dans l'idée d'étudier les plongements d'un système dynamique dans un autre, on étudie ici les plongements d'un espace topologique X dans un espace topologique Y. Ceci pour deux raisons : si (X,T) se plonge dans (Y,S), alors X se plonge topologique dans Y. De plus, dans la section 3., on cherche un équivalent pour les systèmes dynamiques pour \ref{Menger} ci-dessous.

\begin{propo}\label{plonge}
Soit X un espace métrisable compact. Alors X se plonge topologiquement dans $[0,1]^\mathbb{N}$.
\end{propo}

\begin{proof}
On va exhiber une suite $(f_n)_{n\in \mathbb{N}}$ de fonctions continues de X dans $[0,1]$ telles que :
\[\forall x\neq y, \exists n\in \mathbb{N}, f_n(x)\neq f_n(y)\]
Alors on pourra considérer l'application continue suivante :
\begin{align*}
\phi : &X \longrightarrow [0,1]^\mathbb{N}& \\
&x\longmapsto (f_n(x))_{n\in \mathbb{N}} &
\end{align*}
qui est injective, donc induit une application bijective de X sur $\phi (X)$, et dont la réciproque $\phi ^{-1}$ est continue car X est compact.

Soit $n\in \mathbb{N}$. X peut être recouvert par un nombre fini de boules fermées de rayon $\frac{1}{n}$. Pour chaque couple $(B_1,B_2)$ de telles boules, disjointes, la fonction :
\begin{align*}
&X \longrightarrow [0,1]& \\
&x\longmapsto \frac{d(x,B_1)}{d(x,B_1)+d(x,B_2)} &
\end{align*}
est continue et vaut 0 sur la boule $B_1$, et 1 sur la boule $B_2$.
On construit ainsi, pour tout $n\in \mathbb{N}$, un nombre fini de fonctions.

Considérons l'ensemble de ces fonctions. C'est un ensemble dénombrable. Soient x et y distincts dans X. Alors il existe $n\in \mathbb{N}$ tel que x et y appartiennent à deux boules disjointes de rayon $\frac{1}{n}$ du recouvrement exhibé ci-dessus. La fonction associée prend des valeurs distinctes (0 et 1) en x et en y.

 \end{proof}

\begin{thm} \label{Menger}
Soit X un espace compact métrisable de dimension topologique finie n. Alors X se plonge topologiquement dans $\mathbb{R}^{2n+1}$.
\end{thm}

On montre ce théorème en utilisant le théorème de Baire. On démontre le résultat plus fort suivant :

\begin{thm}\label{Menger2}
Soit X un espace compact métrisable de dimension topologique finie n. Soit m un entier tel que $m\geq 2n+1$. Alors l'ensemble des applications de $C(X,\mathbb{R} ^m)$ qui induisent un homéomorphisme de X sur son image est un $G_\delta$-dense dans $C(X,\mathbb{R} ^m)$.
\end{thm}

Pour cela, on considèrera les espaces suivants :
Pour $\epsilon>0$, $C_\epsilon (X,\mathbb{R}^m)$ est l'espace des applications de  $C(X,\mathbb{R} ^m)$ qui sont $\epsilon$-injectives, c'est-à-dire telles que :
\[ \forall (x,y) \in X, f(x)=f(y) \Rightarrow d(x,y)<\epsilon\]
On va montrer qu'ils sont tous ouverts denses dans $C(X,\mathbb{R} ^m)$, qu'on a muni de la distance uniforme $d_\infty$.

\begin{Lemme}
Soit $\epsilon>0$ et $m\in \mathbb{N}$. Alors l'espace $C_\epsilon (X,\mathbb{R}^m)$ est ouvert dans $C(X,\mathbb{R} ^m)$.
\end{Lemme}

\begin{proof}
Soit $f\in C_\epsilon (X,\mathbb{R}^m)$. Posons :
\[ K= \{(x,y)\in X\times X \: |\: d(x,y)\geq \epsilon \} \]
Alors :
\[ \forall (x,y)\in K, d(f(y),f(x))>0 \]
Or K est fermé dans $X\times X$ donc compact. Donc il existe $\delta >0$ tel que :
\[ \forall (x,y)\in K, d(f(y),f(x))\geq \delta \]
Soit maintenant $g\in C(X,\mathbb{R} ^m)$ telle que $d_\infty(f,g)\leq \frac{\delta}{4}$. Alors pour tout $(x,y)\in K$,
\[ \delta \leq d(f(y),f(x))\leq d(f(y),g(y))+d(g(y),g(x))+d(g(x),f(x))\leq d(g(y),g(x))+\frac{\delta}{2} \]
Donc $g\in C_\epsilon (X,\mathbb{R}^m)$.
Finalement $C_\epsilon (X,\mathbb{R}^m)$ est ouvert dans $C (X,\mathbb{R}^m)$
\end{proof}

Pour montrer la densité des espaces $C_\epsilon (X,\mathbb{R}^m)$, nous aurons besoin d'un lemme intermédiaire.

On rappelle la définition suivante :
\begin{Def}
Soit $\{ v_1, v_2,..., v_r\}$ une famille de vecteurs de $\mathbb{R}^m$. Alors on dit qu'elle est affinement indépendante si $\{v_2-v_1,v_3-v_1,...,v_{r}-v_1\}$ est une famille libre.
\end{Def}

\begin{Def}
Soit $m\in \mathbb{N}$ et $r\in \mathbb{N}$. Soit $\{ p_0, p_1,...,p_r\} $ une famille de points de $\mathbb{R}^m$. On dit qu'elle est en position générale si toute sous-famille de $\{ p_0, p_1,...,p_r\} $  comprenant au plus m+1 points est affinement indépendante.
\end{Def}

\begin{Lemme}
Soient $m\in \mathbb{N}$ et $r\in \mathbb{N}$. Soient $\{ p_0, p_1,...,p_r\} $ une famille de points de $\mathbb{R}^m$ et $\epsilon>0$. Alors il existe $\{ q_0, q_1,...,q_r\} $ une famille de points de $\mathbb{R}^m$ en position générale telle que
\[ \forall 0\leq i\leq r, d(p_i,q_i)<\epsilon \]
\end{Lemme}

\begin{proof}
Montrons-le par récurrence sur r.
Un singleton est toujours en position générale, on peut prendre $q_0=p_0$.
Supposons $q_0, q_1,...,q_{r-1}$ déjà construits. Soit une sous-famille de $\{ q_0, q_1,...,q_{r-1}\}$ comprenant au plus m points. Alors elle engendre un sous-espace affine de $\mathbb{R}^m$ de dimension strictement inférieure à m. Cet espace est donc d'intérieur vide dans $\mathbb{R}^m$. Soit F l'union de tous les espaces de ce type. Alors F est une union finie de fermés d'intérieur vide, donc d'après le théorème de Baire, F est d'intérieur vide. On peut donc trouver un point $q_r$ dans son complémentaire tel que $d(q_r,p_r)\leq \epsilon$. Alors par définition de F, $\{ q_0, q_1,...,q_r\}$ est en position générale.
\end{proof}

\begin{Lemme}
Soit X un espace compact métrisable de dimension topologique finie n. Soit m un entier tel que $m\geq 2n+1$ et $\epsilon>0$. Alors l'espace $C_\epsilon (X,\mathbb{R}^m)$ est dense dans $C(X,\mathbb{R} ^m)$.
\end{Lemme}

\begin{proof}
Soit $f\in C(X,\mathbb{R}^m)$ et $\delta >0$. Comme X est compact, f est uniformément continue. Soit $\eta>0$ tel que si $d(x,y)\leq \eta$, alors $d(f(x),f(y))\leq \frac{\delta}{2}$. D'après \ref{maille}, on peut trouver un recouvrement ouvert fini $\alpha$ de X tel que $ord(\alpha)=n$ et $maille(\alpha)\leq min(\eta ,\epsilon)$. Notons $\alpha =\{ U_0, ...,U_r\}$ et choisissons pour tout $U_i$ un point $a_i\in U_i$. Alors en posant $p_i=f(a_i)$ on obtient $\{ p_0, p_1,...,p_r\} $ une famille de points de $\mathbb{R}^m$ à laquelle on peut appliquer le lemme précédent. On obtient  une famille de points $\{ q_0,q_1,..,q_r\}$ en position générale telle que
\[ \forall 0\leq i\leq r, d(p_i,q_i)<\frac{\delta}{2} \]
Soit $(\rho_i)_{i\in \{0,1...,r\}}$ une partition de l'unité subordonnée à $\alpha$. Alors on peut définir l'application :
\begin{align*}
g:&X\longrightarrow \mathbb{R}^m\\
&x\longmapsto \sum_{i=0}^r\rho_i(x)q_i
\end{align*}

g est une fonction continue. Montrons qu'elle est $\epsilon$-injective. Soient x et y tels que g(x)=g(y). Posons $I_x=\{ 0\leq i\leq r \: | \: \rho_i(x)>0\}$ et de même $I_y=\{ 0\leq i\leq r \: | \: \rho_i(y)>0\}$. On a :
\[ \sum_{i\in I_x}\rho_i(x)q_i-\sum_{i\in I_y}\rho_i(y)q_i =0\]
Comme les coefficients en jeu dans cette combinaison linéaire sont tous non nuls, la famille $(q_i)_{i\in I_x\cup I_y}$ est liée. Comme $\{ q_0,q_1,..,q_r\}$ est en position générale, on en déduit que $card(I_x\cup I_y)>m+1$. Or $card(I_x)+card(I_y)\leq 2ord(\alpha)+2\leq 2n+2\leq m+1$. On en déduit que $I_x$ et $I_y$ ont un élément i en commun. Alors x et y sont dans $U_i$. Comme $maille(\alpha)<\epsilon$, $d(x,y)<\epsilon$. Finalement, g est $\epsilon$-injective.

Montrons que $\Arrowvert f-g \Arrowvert _\infty \leq \delta$. Soit $x\in X$. Si $\rho_i(x)>0$, alors comme $diam(U_i)\leq \eta$, $d(f(x),p_i)\leq \frac{\delta}{2}$. Donc $d(f(x),\sum_{i=0}^r\rho_ip_i)\leq \frac{\delta}{2}$.
Alors on a :
\begin{align*}
d(f(x),g(x))&\leq \frac{\delta}{2} + d(\sum_{i=0}^r\rho_i(x)q_i,\sum_{i=0}^r\rho_i(x)p_i)\\
&\leq \frac{\delta}{2} + d(\sum_{i=0}^r\rho_i(x)(q_i-p_i),0)\\
&\leq \frac{\delta}{2} + \frac{\delta}{2}\times \sum_{i=0}^r\rho_i(x)\leq \delta
\end{align*}
Donc $\Arrowvert f-g \Arrowvert _\infty \leq \delta$.

Finalement, on a montré que $C_\epsilon (X,\mathbb{R}^m)$ est dense dans $C(X,\mathbb{R} ^m)$.
\end{proof}

Démontrons maintenant le théorème.

\begin{proof}
L'ensemble des applications de $C(X,\mathbb{R} ^m)$ qui induisent un homéomorphisme de X sur son image s'écrit de la façon suivante :
\[ \bigcap_{k=1}^\infty C_{\frac{1}{k}} (X,\mathbb{R}^m)  \]
D'après les lemmes précédents, c'est donc une intersection dénombrable d'ouverts denses de l'espace métrique complet $C(X,\mathbb{R}^m)$. On peut donc appliquer le théorème de Baire pour conclure.
\end{proof}

\section{Moyenne dimension topologique}

La moyenne dimension topologique est un invariant topologique, introduit par Gromov, et qui permet de distinguer entre des systèmes de dimension infinie.

Dans toute cette section, X est un espace métrique compact et T un homéomorphisme sur X.

\subsection{Généralités}

\begin{Notation}
Soit $\alpha$ un recouvrement ouvert de X. Soient a et b dans $\mathbb{Z}$ tels que a<b. On note :
\[ \alpha_a^b=T^{-a}(\alpha)\vee T^{-a-1}(\alpha)\vee ... \vee T^{-b}(\alpha)\]
$\alpha_a^b$ est un recouvrement ouvert de X.
\end{Notation}

\begin{propo}
Pour tout recouvrement ouvert fini $\alpha$ de X, la suite $(D(\alpha_0^{n-1}))_{n\in \mathbb{N}}$ est sous-additive.
\end{propo}

\begin{proof}
Soient m et n dans $\mathbb{N}$. Alors :
\[ \alpha_0^{n+m-1} = \alpha_0^{n-1}\vee \alpha_n^{n+m-1}=\alpha_0^{n-1}\vee T^{-n}(\alpha_0^{m-1})\]
Donc d'après la proposition précédente et \ref{reciproque} :
\[ D(\alpha_0^{n+m-1})\leq D(\alpha_0^{n-1})+D(T^{-n}(\alpha_0^{m-1}))\leq D(\alpha_0^{n-1})+D(\alpha_0^{m-1})\]
\end{proof}

On en déduit que pour tout recouvrement ouvert fini $\alpha$ de X, la suite
$(\frac{D(\alpha_0^{n-1})}{n})_{n\in \mathbb{N}}$ admet une limite, et de plus :

\[ \lim_{n\to +\infty}\frac{D(\alpha_0^{n-1})}{n}=\inf_{n\in \mathbb{N}}\frac{D(\alpha_0^{n-1})}{n} \]

\begin{Def}
On définit la dimension moyenne du système dynamique (X,T):
\[ mdim(X,T)=\sup_\alpha \lim_{n\to \infty }\frac{D(\alpha_0^{n-1})}{n} \]
où $\alpha$ parcourt tous les recouvrements ouverts finis de X.
\end{Def}

On démontre ici quelques propriétés de la moyenne dimension. On remarquera des similitudes avec la dimension topologique, dans les énoncés et la manière de les démontrer.

\begin{propo}
 Si (X,T) et (Y,S) sont conjugués, alors $mdim(X,T)=mdim(Y,S)$.
\end{propo}

\begin{proof}
Soit $h:X\longrightarrow Y$ un homéomorphisme tel que $h\circ T=S\circ h$. Soit $\alpha$ un recouvrement ouvert fini de Y et $\tilde{\alpha}=h^{-1}(\alpha)$. Soit $n\in \mathbb{N}$. On a $\tilde{\alpha}_0^{n-1}=h^{-1}(\alpha_0^{n-1})$, donc d'après \ref{reciproque}, $D(\alpha_0^{n-1})\geq D(\tilde{\alpha}_0^{n-1} )$. On obtient :$mdim(X,T)\leq \frac{D(\alpha_0^{n-1})}{n}$, donc finalement $mdim(X,T)\leq mdim(Y,S)$. En reprenant la même démonstration avec $h^{-1}$, on obtient $mdim(X,T)\geq mdim(Y,S)$.
Finalement $mdim(X,T)=mdim(Y,S)$.
\end{proof}

\begin{propo}\label{sousesp}
Si F est un sous-ensemble fermé de X invariant par T, alors $mdim(F,T|_F)\leq mdim(X,T)$.
\end{propo}

\begin{proof}
Soit $\alpha$ un recouvrement ouvert fini de F. Il existe une famille d'ouverts de X, $({U_i})_{i\in I}$, telle que : $F\subset \bigcup_{i\in I}U_i$. Comme $X\backslash F$ est un ouvert de X, $\beta =({U_i})_{i\in I}\cup (X\backslash F)$ est un recouvrement ouvert fini de X. Soit $n\in \mathbb{N}$ et $\gamma$ un recouvrement ouvert fini de X raffinant $\beta_0^{n-1}$ tel que $ord(\gamma)=D(\beta_0^{n-1})$. Soit $\tilde{\gamma}=(U\cap F)_{U\in \gamma}$. Alors $\tilde{\gamma}$ est un recouvrement ouvert fini de F raffinant $\alpha_0^{n-1}$, et tel que $ord(\tilde{\gamma})\leq ord(\gamma)$. Donc :
\[ord(\alpha_0^{n-1})\leq ord(\tilde{\gamma})\leq ord(\gamma)=D(\beta)\] Finalement, $mdim(F,T|_F)\leq mdim(X,T)$.
\end{proof}

\begin{propo}
Soit $m\in \mathbb{N}$. Alors $mdim(X,T^m)=m\times mdim(X,T)$.
\end{propo}

\begin{proof}
Soit $\alpha$ un recouvrement ouvert fini de X. Alors pour tout $n\in \mathbb{N}$ :
\[ lim_{n\to +\infty}\frac{D(\alpha \vee T^{-m}...\vee T^{-m(n-1)} )}{n}\leq m\times lim_{n\to +\infty}\frac{D(\alpha_0^{mn-1})}{nm} \]
On en déduit $mdim(X,T^m)\leq m\times mdim(X,T)$ D'autre part,
\[ \alpha_0^{mn-1}=\alpha_0^{k-1}\vee T^{-k}\alpha_0^{k-1} ... \vee T^{-(n-1)m}\alpha_0^{n-1} \]
D'où : $\frac{D(\alpha_0^{mn-1})}{mn}\leq \frac{mdim(X,T^m)}{m}$.
Donc $mdim(X,T)\leq \frac{mdim(X,T^m)}{m} $. Finalement $mdim(X,T^m)=m\times mdim(X,T)$.
\end{proof}

\begin{propo}
Soit $I\subset \mathbb{N}$, et $(X_i,T_i)_{i\in I}$ une famille de systèmes dynamiques (finie ou non). Alors :
\[ mdim(\prod_{i\in I}X_i,\prod_{i\in I}T_i)\leq \sum_{i\in I}mdim(X_i,T_i) \]
\end{propo}

\begin{proof}
Soit $\alpha$ un recouvrement ouvert fini de $X=\prod_{i\in I}X_i$. Tout ouvert U dans $\alpha$ peut être décrit comme l'union d'ouverts de la forme $U_1\times U_2\times ...\times U_{n_i}$, où chaque $U_i$ est un ouvert de $X_i$. On peut donc trouver un entier N et un recouvrement ouvert fini de X, $\beta$, raffinant $\alpha$, de la forme :
\[ \beta =\pi_1^{-1}(\beta_1)\vee ... \vee \pi_N^{-1}(\beta_N) \]
où les $\beta_i$ sont des recouvrements ouverts finis des $X_i$, et $\pi_i :X \to X_i$ la projection sur la i-ième coordonnée.
Alors, d'après \ref{joint}, pour tout $n\in \mathbb{N}$,
\[D(\alpha_0^{n-1})\leq D(\beta_0^{n-1})\leq \sum_{i=1}^N D((\beta_i)_0^{n-1})  \]
On en déduit :
\[ \lim_{n\to \infty }\frac{D(\alpha_0^{n-1})}{n}\leq \sum_{i=1}^N \lim_{n\to \infty }\frac{D((\beta_j)_0^{n-1})}{n}\leq \sum_{i=1}^Nmdim(X_i,T_i)\]
\end{proof}

\subsection{Exemples}

D'après la proposition qui suit, la notion de moyenne dimension topologique ne devient pertinente qu'en dimension topologique infinie.

\begin{propo}
Supposons X de dimension topologique finie. Soit $T:X\to X$ une application continue. Alors $mdim(X,T)=0$.
\end{propo}
\begin{proof}
Soit N la dimension topologique de X. Soit $\alpha$ un recouvrement ouvert fini de X. Alors $D(\alpha_0^{n-1})\leq N$. On en déduit que $\lim_{n\to +\infty} \frac{D(\alpha_0^{n-1})}{n}=0$, donc que $mdim(X,T)=0$.
\end{proof}

\begin{Ex}
Si $(X_i,T_i)_{i\in \mathbb{N}}$ est une suite de systèmes dynamiques, tous de dimension topologique finie, alors $mdim(\prod_{i\in I}X_i,\prod_{i\in I}T_i)=0$
\end{Ex}

Dans la fin de cette section, on calcule la moyenne dimension de certains décalages. On va obtenir des premiers exemples, simples, de systèmes dynamiques de moyenne dimension n pour tout $n\in \mathbb{N}$. Mais contrairement à la dimension topologique, la moyenne dimension peut prendre toutes les valeurs réelles positives, ce qu'on verra dans la section suivante.

Dans tout la suite, pour tout compact K, on notera $\sigma$ le décalage bilatéral sur l'ensemble des suites $K^\mathbb{Z}$.

De plus, pour $I\subset \mathbb{N}$, $\pi_{I}$ désignera la projection qui à un élément de $K^\mathbb{Z}$ associe la liste dans $K^{card(I)}$ de ses coordonnées indicées par I.\\

\begin{propo}
Soit K un compact de dimension topologique finie d. Alors $mdim(K^\mathbb{Z},\sigma)\leq d$.
\end{propo}

\begin{proof}
Soit $\tilde{\alpha}$ un recouvrement ouvert fini de $K^\mathbb{Z}$.

Par définition de la topologie produit, $\tilde{\alpha}$ admet un sous-recouvrement par des ouverts de la forme $\pi _{-k,...,k}^{-1}(V)$, où les V sont des ouverts de $K^{2k+1}$.

Par compacité de $K^\mathbb{Z}$, on peut donc trouver $N\in \mathbb{N}$ et un sous-recouvrement fini $\alpha$ de $\tilde{\alpha}$ par des ouverts de la forme $\pi _{\{ -N,...,N\} }^{-1}(V)$, où les V sont des ouverts de $K^{2N+1}$.

Pour $n\in \mathbb{N}$, $\alpha_0^{n-1}$ est de la forme $\pi _{-N,...,N+n-1}^{-1}(\beta)$ où $\beta$ est un recouvrement ouvert fini de $\pi _{\{ -N,...,N+n-1\} }(K^\mathbb{Z})\subset K^{2N+1+n}$. Donc d'après \ref{reciproque}, $D(\alpha_0^{n-1})\leq D(\beta)$. Donc :

\[D(\tilde{\alpha}_0^{n-1})\leq D(\alpha_0^{n-1})\leq D(\beta)\leq d\times (2N+1+n)\]

On en déduit :
\[\frac{D(\tilde{\alpha}_0^{n-1})}{n} \leq \frac{d\times (2N+1+n)}{n}\]
D'où :
\[ \lim_{n\to \infty }\frac{D(\tilde{\alpha}_0^{n-1})}{n} \leq d\]

\end{proof}

\begin{propo}\label{minoration}
Soit X un sous-décalage de $([0,1]^d)^\mathbb{Z}$ et $\theta$ un réel tels qu'il existe un ensemble d'indices $I\subset \mathbb{N}$ tel que :
\begin{enumerate}
\item \[\theta=\limsup_{n\to +\infty} \frac{|I\cap \{0,...,n-1\}|}{n}\]
\item il existe $\bar{x} \in X$ tel que pour tout $x\in ([0,1]^d)^\mathbb{Z}$ tel que
\[ \pi _{\mathbb{Z}\textbackslash I}(x)=\pi _{\mathbb{Z}\textbackslash I}(\bar{x}) \]
alors $x\in X$
\end{enumerate}
Alors $mdim(X,\sigma)\geq \theta d.$

\end{propo}

\begin{proof}
Soit $\tilde{\alpha}$ un recouvrement ouvert fini de $[0,1]^d$, tel qu'aucun de ses éléments n'intersecte deux faces opposées de $[0,1]^d$. Alors $\pi_{1}^{-1}(\tilde{\alpha})$ induit un recouvrement ouvert fini $\alpha$ de X.
\\Soient $n\in \mathbb{N}$, et $\beta$ un sous-recouvrement arbitraire de $\alpha_{0}^{n-1}$. Posons $I_{n}=I\cap \{0,...,n-1\}$ et :
\[\tilde{\beta}=\{ \pi_{I_{n}}(\{x\in U | \pi _{\mathbb{Z}\textbackslash I}(x)=\pi _{\mathbb{Z}\textbackslash I}(\bar{x}))\}_{U\in \beta } \]

Alors $\tilde{\beta}$ est un recouvrement ouvert fini de $([0,1]^d)^{|I_{n}|}$. En effet, soit $U\in \beta$. Si U est de la forme $\prod_{i\in \mathbb{Z}}U_i$ où chaque $U_i$ est un ouvert de $[0,1]^d$. Alors $\pi_{I_n}(\{x\in U | \pi _{\mathbb{Z}\textbackslash I}(x)=\pi _{\mathbb{Z}\textbackslash I}(\bar{x})\} )$ est ouvert en tant que produit d'ouverts de $([0,1]^d)^{|I_n|}$. Comme tout ouvert de $\beta$ s'écrit comme union d'ouverts de cette forme, on en déduit que tous les éléments de $\tilde{\beta}$ sont ouverts.

De plus, comme aucun élément de $\alpha$ n'intersecte deux faces opposées de $[0,1]^d$, si $V\in \alpha_0^{n-1}$, $\pi_{\{ 0,...,n-1 \}}(V)$ n'intersecte pas deux faces opposées de $([0,1]^d)^n$. Donc a fortiori, aucun élément de $\tilde{\beta}$ n'intersecte deux faces opposées de $[0,1]^{d|I_n|}$. Donc d'après \ref{Lebesgue} : $ord(\tilde{\beta})\geq d|I_{n}|$.

D'autre part, $ord(\beta)\geq ord(\tilde{\beta})$. En effet, soit $x'\in ([0,1]^d)^{|I_n|}$. Alors si $x'\in V\in \tilde{\beta}$, il existe $x\in U\in \beta$ tel que $x'=\pi_{I_n}(x)$. Donc $ord_{x'} (\tilde{\beta}) \leq ord_x(\beta)\leq ord(\beta)+1$. Finalement, $ord(\beta)\geq ord(\tilde{\beta})$.

On en déduit que $ord(\beta)\geq d|I_{n}|$, et donc que $D(\alpha_{0}^{n-1})\geq d|I_{n}|$.
Finalement :
\[mdim(X) \geq \lim_{n\to+\infty} \frac{D(\alpha_{0}^{n-1})}{n}\geq d\limsup_{n\to +\infty} \frac{|I\cap \{0,...,n-1\}|}{n}=\theta d\]
\end{proof}

\begin{coro}
Pour tout $d\in \mathbb{N}$, $mdim(([0,1]^d)^\mathbb{Z},\sigma)=d$.
\end{coro}

On remarque que pour démontrer ce résultat, on a procédé de manière analogue à la démonstration de $dim([0,1]^d)=d$. On a notamment la même utilisation du lemme de Lebesque.

\begin{coro}
$mdim(([0,1]^\mathbb{N})^\mathbb{Z},\sigma)=\infty$.
\end{coro}
\begin{proof}
Soit $n\in \mathbb{N}$. Posons $A=\{ (u_k)_{k\in \mathbb{N}}\: | \: \forall k\geq n, u_k=0\}$. Alors $A^\mathbb{N}$ est un sous-ensemble fermé de $[0,1]^\mathbb{Z}$, $\sigma$-invariant, et homéomorphe à $([0,1]^n)^\mathbb{Z}$. Donc d'après \ref{sousesp}, $n=mdim(A^\mathbb{Z},\sigma)\leq mdim(([0,1]^\mathbb{N})^\mathbb{Z},\sigma)$. Ceci est vrai pour tout $n\in \mathbb{N}$, donc $mdim(([0,1]^\mathbb{N})^\mathbb{Z},\sigma)=\infty$.
\end{proof}

\subsection{Systèmes minimaux de moyenne dimension r donnée}

Notre but est de démontrer la proposition suivante :

\begin{propo}\label{minimal}
Soit un réel $r\geq 0$. Il existe $X\subset [0,1]^\mathbb{Z}$ et $m\in \mathbb{Z}$ tel que $mdim(X,\sigma^m)=r$ et tel que $(X,\sigma^m)$ est minimal.
\end{propo}

Dans la suite, pour $I\subset \mathbb{N}$, $\pi_{I}$ désignera la projection qui à un élément de $[0,1]^\mathbb{Z}$ associe la liste dans $[0,1]^{card(I)}$ de ses coordonnées indicées par I.\\

Soit un réel $r> 0$. Soient t un réel et $m\in \mathbb{N}$ tels que $\frac{r}{m}=t$ et $0< t<1$. Alors s'il existe un sous-décalage $X\subset [0,1]^\mathbb{Z}$ tel que $mdim(X,\sigma)=t$, on a $mdim(X,\sigma^m)=m\times mdim(X,\sigma)=r$. On peut donc supposer dans la suite $r<1$.\\

On va définir une suite de sous-décalages $X_n$ de type bloc associés à $(q_n,B_n)$, où $B_n$ est un sous-ensemble fermé de $[0,1]^{q_n}$.
$X_n$ est l'ensemble des suites de $[0,1]^\mathbb{N}$ qui peuvent s'écrire comme concaténation d'éléments appartenant à $B_n$.
On pose pour tout $n\in \mathbb{N}$:
\[ \Phi_n = \{ (x_i)_{i\in \mathbb{Z}} \: | \: (x_i,...,x_{i+q_n-1}) \in B_n \: si \: i\equiv 0\:  mod \: q_n \} \]

Le sous-décalage recherché sera :
\[X=\bigcap _{n\in \mathbb{N}} X_{n}\]\\

On va construire la suite $(X_n)_{n\in \mathbb{N}}$ de manière à utiliser \ref{minoration} pour minorer mdim(X,T) par r, et le lemme suivant pour majorer mdim(X,T) par r :
Ce lemme est une version simplifiée de la proposition 4.1 dans \cite{K}.

\begin{Lemme}\label{majoration}
Soit un sous-décalage $X\subset [0,1]^\mathbb{Z}$. Alors
\[mdim(X,\sigma)\leq \liminf_{n \to +\infty} \frac{dim(\pi_{n}(X))}{n}\]
où $\pi_{n}:[0,1]^\mathbb{Z}\to [0,1]^n$ est la projection qui à un élément de $[0,1]^\mathbb{Z}$ associe ses n premières coordonnées.
\end{Lemme}

\begin{proof}

Soit $\tilde{\alpha}$ un recouvrement ouvert fini arbitraire de X.

Par définition de la topologie produit, tout ouvert de $\tilde{\alpha}$ est une union d'ouverts de la forme $\pi _{-k,...,k}^{-1}(V)$, où les V sont des ouverts de $[0,1]^{2k+1}$.

Par compacité de X, on peut donc trouver $N\in \mathbb{N}$ et un recouvrement ouvert fini $\alpha$ tel que $\alpha \succ \tilde{\alpha}$ composé d'ouverts de la forme $\pi _{\{ -N,...,N\} }^{-1}(V)$, où les V sont des ouverts de $[0,1]^{2N+1}$.

$\alpha_0^{n-1}$ est de la forme $\pi _{-N,...,N+n-1}^{-1}(\beta)$ où $\beta$ est un recouvrement ouvert fini de $\pi _{\{ -N,...,N+n-1\} }(X)$. Donc d'après \ref{reciproque}, $D(\alpha_0^{n-1})\leq D(\beta)$.

Donc :
\[D(\tilde{\alpha}_0^{n-1})\leq D(\alpha_0^{n-1})\leq D(\beta)\leq dim( \pi _{\{ -N,...,N+n-1\} }(X)))\]

De plus, $\pi _{\{ -N,...,N+n-1\} }(X)\subset \pi _{\{ -N,...,N-1\} }(X)\times \pi _{n}(X)$. Donc d'après \ref{produit} :
\[ dim( \pi _{\{ -N,...,N+n-1\} }(X))\leq dim(\pi _{\{ -N,...,N-1\} }(X))+dim(\pi _{n}(X))\]
Donc $D(\tilde{\alpha}_0^{n-1})\leq dim(\pi _{\{ -N,...,N-1\} }(X))+dim(\pi _{n}(X))$ et en divisant par n de chaque côté, on en déduit :
\[\lim_{n\to +\infty}\frac{D(\tilde{\alpha}_0^{n-1})}{n}\leq \liminf_{n \to +\infty} \frac{dim(\pi_{n}(X))}{n} \]
D'où :
\[mdim(X,\sigma)\leq \liminf_{n \to +\infty} \frac{dim(\pi_{n}(X))}{n}\]
\end{proof}

Démontrons maintenant \ref{minimal}.

\begin{proof}

Pour construire la suite $(X_n)_{n\in \mathbb{N}}$, on procède par récurrence.

Posons $q_0=0$, $B_0=[0,1]$ et $I_0=\mathbb{N}$.

Supposons que pour un certain $n\in \mathbb{N}$ on ait déjà construit $q_n$ et $B_n$.
$(X_n,\sigma)$ est un sous-décalage de type bloc, il existe donc pour ce système un élément y dont l'orbite est dense dans $X_n$.
Quitte à composer plusieurs fois y par $\sigma$, on peut supposer : $\pi_{\{ 0,...,q_n-1\} }(y)\in B_n$.
D'autre part, il existe un entier $r_{n+1}$ assez grand pour que :
\[ \pi_{\{ -r_{n+1},...,r_{n+1}\} }(x)=\pi_{\{- r_{n+1},...,r_{n+1}\} }(x') \Rightarrow d(x,x')\leq 2^{-n}\]

On peut alors trouver un entier $L_{n+1}\geq r_{n+1}$, multiple de $q_n$, tel que pour tout $x\in X_n$, il existe $k\in \{ -L_{n+1}+r_{n+1}, ..., L_{n+1}-r_{n+1}\}$ tel que :
\[ d(x,\sigma ^k(y))\leq 2^{-n} \]

Posons $q_{n+1}=2L_{n+1}(a_{n+1}+1)$, où $a_n$ est un entier qui reste à déterminer.
On définit enfin $B_{n+1}$ comme l'ensemble des concaténations de $\frac{q_{n+1}}{q_n}$ blocs appartenant à $B_n$, se terminant par le $2L_{n+1}-uplet : (y_{-L_{n+1}}, ..., y_{L_{n+1}-1}\} $.

L'orbite de tout point de $X_{n+1}$ est $2^{-(n+1)}-dense$. En effet, soit $x' \in X_{n+1}$ et $x \in X_{n+1}$. Alors il existe $k\in \{ -L_{n+1}+r_{n+1}, ..., L_{n+1}-r_{n+1}\}$ tel que :
\[ d(x',\sigma ^k(y))\leq 2^{-n} \]
D'autre part, comme $k\in \{ -L_{n+1}+r_{n+1}, ..., L_{n+1}-r_{n+1}\}$ et $x \in X_{n+1}$, on a  $\pi_{\{ -r_{n+1},...,r_{n+1}\} }(\sigma ^k(x))=\pi_{\{- r_{n+1},...,r_{n+1}\} }(\sigma ^k(y))$, donc $d(\sigma ^k(x), \sigma ^k(y))\leq 2^{-n}$. Donc $d(x', \sigma ^k(x))\leq 2^{-(n+1)}$.
On en déduit que l'orbite de tout point de X est dense dans X, c'est-à-dire que X est minimal.\\

De plus, on va construire pour tout $n\in \mathbb{N}$ un sous-ensemble $I_n\subset \mathbb{Z}$ et un élément $x_n\in X_n$, vérifiant les propriétés suivantes :
\begin{enumerate}
\item la suite $(I_n)_{n\in \mathbb{N}}$ est décroissante
\item $\pi_{\mathbb{Z} \backslash I_n}(x)=\pi_{\mathbb{Z} \backslash I_n}(x_n) \Rightarrow x\in X_n$
\item $\forall m\geq n, \pi_{\mathbb{Z} \backslash I_n}(x_m)=\pi_{\mathbb{Z} \backslash I_n}(x_n)$
\end{enumerate}

On aura alors qu'en définissant $I=\bigcap _{n\in \mathbb{N}} I_{n}$ et $X=\bigcap _{n\in \mathbb{N}} X_{n}$:
\[ mdim(X,\sigma)\geq \limsup_{n\to +\infty} \frac{|I\cap \{0,...,n-1\}|}{n}\]
En effet, on peut extraire de la suite $(x_n)_{n\in \mathbb{N}}$ une sous-suite convergeant vers $ \bar{x} \in X$. Alors, comme pour tout $n\in \mathbb{N}$, $\pi_{\mathbb{Z} \backslash I_n}(x_n)=\pi_{\mathbb{Z} \backslash I_n}(\bar{x})$, on a :
\[\pi_{\mathbb{Z} \backslash I}(\bar{x})=\pi_{\mathbb{Z} \backslash I}(x)\Rightarrow x\in X\]
On peut donc appliquer le \ref{majoration}.\\

Construisons ces éléments pour tout $n\in \mathbb{N}$.

Pour tout $n\in \mathbb{N}$, posons $R_n$ l'ensemble des entiers $i\in \mathbb{Z}$ tels qu'il existe $k\in \{ 0,...,q_n-2L_n\}$ tel que $i\equiv k\: mod \: q_n$.
On pose alors $I_n=\bigcap_{i=0}^n R_i$ et $x_n$ est un élément arbitraire de $\Phi_n$. \\

Vérifions qu'on obtient les propriétés voulues.

La suite $(I_n)_{n\in \mathbb{N}}$ est décroissante.

Soit x tel que $\pi_{\mathbb{Z} \backslash I_n}(x)=\pi_{\mathbb{Z} \backslash I_n}(x_n)$. Alors pour tout $l\leq n$, si i est tel qu'il existe $k\in \{ q_l-2L_l+1,...,q_l\}$ tel que $ i\equiv k \: mod \: q_l$, alors la i-ème coordonnée de x est égale à celle de $x_n$ : c'est $y_{k-q_l+L_l-1}$. Par définition des $B_l$, on en déduit : $x\in \Phi_n$. Donc $x\in X_n$.

Enfin, soit $m\geq n$. $x_m$ appartient à $\Phi_m$, donc à $\Phi_n$. Donc par un raisonnement analogue, $\pi_{\mathbb{Z} \backslash I_n}(x_m)=\pi_{\mathbb{Z} \backslash I_n}(x_n)$.\\

Calculons la densité supérieure de I.

Soit $n\in \mathbb{N}$. Alors pour $m>n$, on a $q_n\leq q_m-2L_m$ car $q_n\:  |\: L_m \: |\: q_m-2L_m$. Donc $\{ 0,...,q_n-1\} \subset R_m$. Donc :
\[ I\cap \{0,...,q_n-1\} = \bigcap_{k=1}^n R_k \]

Donc :
\[ \frac{card(I\cap \{0,...,q_n-1\})}{q_n}=\prod_{k=1}^n \frac{q_k-2L_k}{q_k}\]
Et comme pour tout $1\leq k\leq n$, $q_k=2L_k(a_k+1)$ par définition, on a :
\[ \frac{card(I\cap \{0,...,q_n-1\})}{q_n}=\prod_{k=1}^n \frac{2a_kL_k}{2a_kL_k+2L_k}=(\prod_{k=1}^n 1+\frac{1}{a_k})^{-1}\]

Comme 0<r<1, on peut choisir les $a_n$ de manière à ce que $\sum_{n=1}^\infty log(1+\frac{1}{a_n})=-log(r)$. On a alors :
\[ \frac{card(I\cap \{0,...,q_n-1\})}{q_n}=(\prod_{k=1}^n 1+\frac{1}{a_k})^{-1}=r \]
Donc :
\[ \limsup_{n\to +\infty} \frac{|I\cap \{0,...,n-1\}|}{n}\geq r\]

Finalement, $mdim(X,T)\geq r$.\\

Enfin, soit $n\in \mathbb{N}$. Alors $mdim(X,\sigma)\leq mdim(X_{n},\sigma)$, et d'après \ref{majoration}, \[mdim(X_{n},\sigma)\leq \liminf_{N \to +\infty} \frac{dim(\pi_{N}(X_{n}))}{N}.\]
Or, pour tout $k\in \mathbb{N}$, par définition de $X_n$,
\[\pi_{q_nk}(X_{n})\subset \bigcup _{0\leq i\leq q_n}([0,1]^i\times B_{n}^{k-1}\times [0,1]^{q_n-i}) \]
Pour tout $i\in \mathbb{N}$,
\[ dim([0,1]^i\times B_{n}^{k-1}\times [0,1]^{q_n-i})\leq q_n+dim(B_n^{k-1})=q_n+(k-1)\times dim(B_n)\]
Or on a la relation de récurrence suivante :
\[ \forall n\in \mathbb{N}, dim(B_{n+1})=dim(B_n)\times \frac{q_{n+1}-2L_{n+1}}{q_n}=dim(B_n)\times \frac{a_{n+1}\times 2L_{n+1}}{q_n}\]
Ce qui permet de calculer :
\[ \forall n\in \mathbb{N}, \frac{dim(B_n)}{q_n}=\prod_{i=1}^{n}\frac{a_i}{a_i+1} \]

Donc
\[\frac{dim(\pi_{q_nk}(X_{n}))}{q_n k}\leq \frac{1}{k}+\frac{k-1}{k}\times \prod_{i=1}^{n} \frac{a_i}{a_i+1} \]
Donc :
\[mdim(X_{n},\sigma)\leq \liminf_{k \to +\infty} \frac{\pi_{q_nk}(X_{n})}{q_nk}\leq \prod_{i=1}^{n}\frac{a_i}{a_i+1}.\]
Donc $mdim(X,\sigma)\leq r$.

\end{proof}

\section{Plongements dans des décalages}

Dans ce chapitre, sauf mention contraire, X et Y seront des espaces compacts métrisables, T et S des homéomorphismes, respectivement sur X et Y. On se donne d un entier, et on cherche à savoir quand un système dynamique (X,T) peut se plonger dans le décalage $(([0,1]^d)^\mathbb{Z},\sigma)$.

\subsection{Généralités, premiers exemples}

On rappelle qu'on dit qu'une application $\phi:X\to Y$ est un plongement si elle induit un homéomorphisme de X sur $\phi(X)$.

\begin{Def}
Soient X et Y des espaces topologiques, T un homéomorphisme de X, S un homéomorphisme de Y. On dit que le système dynamique (X,T) se plonge dans le système dynamique (Y,S) s'il existe un plongement $\phi :X\longrightarrow Y$ tel que $\phi \circ T=S\circ \phi$.
\end{Def}

\begin{propo}
Pour que le système dynamique (X,T) se plonge dans le décalage $(([0,1]^d)^\mathbb{Z},\sigma)$, il faut que $mdim(X,T)\leq d$.
\end{propo}

\begin{proof}
C'est une conséquence directe de \ref{sousesp}.
\end{proof}

\begin{Rem}
Le système dynamique (X,T) peut toujours se plonger dans le décalage $(([0,1]^\mathbb{N})^\mathbb{Z}, \sigma)$. En effet, d'après \ref{plonge}, il existe un plongement $\phi$ de X dans $[0,1]^\mathbb{N}$. Alors on peut construire le plongement suivant :
\[ \Phi:x \longmapsto (...,\phi (T^{-1}(x),\phi (x),\phi (T(x)),...)\]
de (X,T) dans $(([0,1]^\mathbb{N})^\mathbb{Z}, \sigma)$.

D'autre part, si on suppose de plus que X est de dimension finie n, alors le système dynamique (X,T) se plonge dans le décalage $(([0,1]^{2n+1})^\mathbb{Z}, \sigma)$. En effet, d'après \ref{Menger}, il existe un plongement $\phi$ de X dans $[0,1]^{2n+1}$, et on peut conclure comme précédemment.
\end{Rem}

La fin de cette section est consacrée à la description d'une méthode utilisée systématiquement dans la suite, pour prouver l'existence de plongements dans des décalages.
Comme dans \ref{Menger} , on utilisera le théorème de Baire. La proposition suivante permet de se ramener à l'espace (métrique complet) $C(X,[0,1]^d)$.

\begin{propo}
Le système dynamique (X,T) se plonge dans le décalage $(([0,1]^d)^\mathbb{Z},\sigma)$ si et seulement s'il existe une application continue $f:X\to [0,1]^d$ telle que :
\[ \forall x\neq y, \exists i\in \mathbb{Z}, f(T^i(x))\neq f(T^i(y))\]
\end{propo}

\begin{proof}

Supposons que $\phi$ est un plongement de X dans $([0,1]^d)^\mathbb{Z}$, tel que $\phi \circ T=\sigma \circ \phi$. Soient x et y distincts dans X. Alors il existe $i\in \mathbb{Z}$ tel que les ièmes coordonnées respectives de $\phi (x)$ et $\phi (y)$ soient distinctes. Comme $\phi \circ T^i=\sigma ^i \circ \phi$, en posant $\pi : ([0,1]^d)^{\mathbb{Z}} \to [0,1]^d$ la projection sur la 0ième coordonnée, on a $\pi \circ \phi \circ T^i(x) \neq \pi \circ \phi \circ T^i(y)$.
En posant $f=\pi \circ \phi$, on a la propriété voulue.

Réciproquement, soit $f\in C(X,[0,1]^d)$ ayant la propriété énoncée. Pour tout x dans X, posons $I_f (x)=(f(T^i(x)))_{i\in \mathbb{Z}}$. L'application $I_f :X\to ([0,1]^d)^{\mathbb{Z}}$ est continue, injective, et $I_f \circ T=\sigma \circ I_f$. Comme X est compact, $I_f$ est donc bien un plongement du système dynamique (X,T) dans le décalage$(([0,1]^d)^\mathbb{Z},\sigma)$.
\end{proof}

On notera désormais, pour $f\in C(X,[0,1]^d)$ :
\[ I_f :x\longmapsto (f(T^i(x)))_{i\in \mathbb{Z}}\]

On pose :
\[ \Delta=\{ (x,x)\: |\: x\in X\}\: ; \: \Omega=(X\times X) \backslash \Delta \]
et on note, pour tout ensemble $K\subset \Omega$, $D_K$ l'ensemble des fonctions continues de X dans $[0,1]^d$ telles que :
\[ \forall (x,y)\in K, \exists i\in \mathbb{Z}, f(T^i(x))\neq f(T^i(y))\]
On cherche donc à montrer que $D_\Omega$ est non vide.
On va chercher, pour tous x et y distincts dans X, un voisinage compact K de (x,y) tel que $D_K$ est dense dans $C(X,[0,1]^d)$.

Alors grâce au lemme suivant, tous les $D_K$ considérés seront ouverts et denses dans $C(X,[0,1]^d)$.

\begin{Lemme}
Soit un ensemble compact $K\subset \Omega$. Alors $D_K$ est ouvert dans $C(X,[0,1]^d)$.
\end{Lemme}

\begin{proof}
Soit $f\in D_K$. Posons :
\begin{align*}
H : &X\times X \to \mathbb{R}& \\
&(x,y)\mapsto sup_{i\in \mathbb{Z}}|f(T^i(x))-f(T^i(y))|&
\end{align*}
Alors :\\
f est continue donc H est semi-continue inférieurement,\\
$\forall (x,y)\in K, H(x,y)>0$,\\
K est compact

On en déduit que H est minorée sur K et atteint sa borne inférieure.

Soit $\epsilon >0$ tel que : $\forall (x,y)\in K, H(x,y)>\epsilon $. Soit $g\in C(X,[0,1]^d)$ telle que $\| f-g\|\leq \frac{\epsilon}{4}.$ Montrons que $g\in D_K$.

Soient $(x,y)\in K$ et $i\in \mathbb{Z}$.
\begin{align*}
 |f(T^i(x))-f(T^i(y))|& \leq |f(T^i(x))-g(T^i(x))|+|g(T^i(x))-g(T^i(y))|+|g(T^i(y))-f(T^i(y))|& \\
 & \leq |g(T^i(x))-g(T^i(y))| +\frac{\epsilon}{2} &
\end{align*}
Donc :
\[ \frac{\epsilon}{2} \leq H(x,y)-\frac{\epsilon}{2} \leq sup_{i\in \mathbb{Z}}|g(T^i(x))-g(T^i(y))| \]
On en déduit que $g \in D_K$. Finalement, $D_K$ est ouvert dans $C(X,[0,1]^d)$.

\end{proof}

Le principal travail dans la preuve sera de trouver des voisinage compacts K tels que $D_K$ est dense. Alors on obtiendra une famille de compacts K recouvrant $\Omega \subset X\times X$. $X\times X$ est compact donc $\Omega$ est un espace de Lindelöf. On pourra donc recouvrir $\Omega$ par un ensemble dénombrable de compacts $K_n$ tels que $D_{K_n}$ est un ouvert dense de $C(X,[0,1]^d)$ pour tout $n\in \mathbb{N}$. Comme $D_\Omega=\bigcap_{n\in \mathbb{N}}D_{K_n}$, d'après le théorème de Baire, $D_\Omega$ sera dense dans $C(X,[0,1]^d)$ donc en particulier non vide. Ce qui permettra de conclure à l'existence d'un plongement.

\subsection{Autour du théorème de plongement de Jaworski}

Nous avons déjà vu, dans la section précédente, que si X est de dimension finie n, alors le système dynamique (X,T) se plonge dans le décalage $(([0,1]^{2n+1})^\mathbb{Z}, \sigma)$. On peut se demander si l'entier $2n+1$ est minimal. Dans le cas d'un système apériodique, Jaworski a montré dans \cite{Jaw} que (X,T) se plongeait dans $([0,1]^\mathbb{Z},\sigma)$. On remarque la similitude, dans l'énoncé et la démonstration, avec \ref{Menger}, qui est d'ailleurs utilisé dans la preuve. Il s'agit ici d'une première tentative pour trouver un équivalent de ce résultat aux plongements de systèmes dynamiques.

\begin{Notation}
Si (X,T) est un système dynamique, et $k\in \mathbb{N}$, on note $P_k$ l'ensemble des points périodiques de (X,T) de période inférieure ou égale à k, et $H_k$ l'ensemble des points périodiques de (X,T) de période k.
Enfin, on note P l'ensemble des points périodiques.
\end{Notation}

\begin{Lemme}\label{dist}
Supposons que (X,T) est de dimension finie n, et que $P_{6n}=\emptyset$. Soient x et y distincts dans X. Alors on peut trouver des indices $i_0,i_1,...,i_{2n}$ tels que les points $T^{i_0}(x),...,T^{i_{2n}}(x),T^{i_0}(y),...,T^{i_{2n}}(y)$ sont tous distincts.
\end{Lemme}

\begin{proof}
Si x et y ne sont pas dans la même orbite, on peut prendre les entiers $\{0,1,2,...,2n\}$ .Sinon, il existe $l\in \mathbb{Z}$ tel que $y=T^lx$. Si x n'est pas périodique, on peut prendre les entiers  $\{0,l+1,2(l+1),...,2n(l+1)\}$. Si x est périodique de période p, alors on construit les indices par récurrence. On peut prendre $i_0$ quelconque. Ensuite, supposons $i_0,i_1,...i_k$ déjà construits. Comme $p>6n$ et $k<2n$, on peut trouver un entier $i_{k+1}$ qui n'appartient pas à l'ensemble :
\[ \{I+p\mathbb{Z}\} \cup \{ I+l+p\mathbb{Z}\} \cup \{ I-l+p\mathbb{Z}\} \]
où $I=\{ i_0, i_1,...,i_k\}$.
\end{proof}

\begin{thm}
Soit X un espace compact métrisable de dimension finie, et T un homéomorphisme sur X sans point périodique. Alors le système dynamique (X,T) se plonge dans le décalage $([0,1]^d,\sigma )$.
\end{thm}

\begin{proof}
On va utiliser la méthode décrite dans la section précédente. Soit $(x_0,y_0)\in \Omega$, et $m= 2dim(X)+1$. Alors d'après \ref{dist}, comme (X,T) est apériodique, on peut trouver des entiers $i_1,i_2,..,i_m$ tels que les points $T^{i_1}(x_0),...,T^{i_m}(x_0),T^{i_1}(y_0),...,T^{i_m}(y_0)$ sont tous distincts. Alors on peut trouver deux voisinages compacts respectifs de $x_0$ et $y_0$, U et V, tels que $T^{i_1}(U),...,T^{i_m}(U),T^{i_1}(V),...,T^{i_m}(V)$ sont deux à deux disjoints. On pose alors $K=U\times V$.

Il suffit de montrer que pour tout K de cette forme, $D_K$ est dense dans $C(X,[0,1])$. Soit $\epsilon >0$ et $f\in  C(X,[0,1])$. On considère l'application $\phi$ définie par :
\begin{align*}
\phi : &U\cup V \longrightarrow \mathbb{R}^m\\
&x\longmapsto (f(T^{i_1}(x),...,f(T^{i_m}(x))
\end{align*}
$\phi$ est une application continue et $m\geq 2dim(X)+1\geq 2dim(U\cup V)+1$ donc d'après \ref{Menger2}, il existe un plongement $\psi :U\cup V \to \mathbb{R}^m$ tel que $\Arrowvert \phi -\psi \Arrowvert|_{\infty, U\cup V}<\epsilon$.

Alors on peut définir sur $Z=\bigcup_{i=1}^mT^i(U) \cup \bigcup_{i=1}^mT^i(V)$ une fonction h telle que :
\[ \forall z \in T^{i}(U)\cup T^{i}(V), h(z)=\psi(T^{-i}(z))|_i\]
Alors h est continue et vérifie : $\Arrowvert h -f|_Z \Arrowvert|_{\infty}<\epsilon$. D'après le théorème d'extension de Tietze, on peut prolonger h en une fonction $\tilde{f}\in C(X,[0,1])$, telle que $\Arrowvert \tilde{f} -f \Arrowvert_{\infty}<\epsilon$. Alors $\tilde{f}$ est dans $D_K$. En effet, si $(x,y)\in K$, par injectivté de $\psi$, il existe $k\leq m$ tel que $\psi(x)|_k \neq \psi(y)|_k$. Donc $\tilde{f}(T^{i_k}(x))\neq \tilde{f}(T^{i_k}(y))$.
Finalement on a prouvé que $D_K$ est dense dans $C(X,[0,1])$.

\end{proof}

On ne peut pas supprimer l'hypothèse de finitude dans le théorème de plongement de Jaworski : d'après \ref{minimal}, il existe un système dynamique minimal (X,T) tel que $mdim(X,T)>1$. \
Ce système ne se plonge pas dans $([0,1]^\mathbb{Z},\sigma)$ (sinon on aurait $mdim(X,T)\leq 1$), et est apériodique. En effet, si x est un point périodique de (X,T), son orbite est finie et dense dans X. Donc X est réduit à cette orbite. Donc (X,T) se plonge dans $([0,1]^\mathbb{Z},\sigma)$. C'est une contradiction.

Il a fallu attendre l'introduction de la moyenne dimension pour construire un tel contre-exemple (dû à E.Lindenstrauss et B.Weiss).

On ne peut pas non plus supprimer l'hypothèse d'apériodicité. Par exemple, considérons la sphère unité $\mathbb{S}^2$ dans $\mathbb{R}^3$, munie de la symétrie par rapport à l'équateur, $s$. Alors si $(\mathbb{S}^2,\tau)$ se plonge dans $([0,1]^\mathbb{Z},\sigma)$, l'ensemble de ses points fixes, qui est un cercle (l'équateur), se plonge dans l'ensemble des points fixes de $([0,1]^\mathbb{Z},\sigma)$, qui est homéomorphe à $[0,1]$. C'est impossible.

Cependant, il est possible d'affaiblir l'hypothèse d'apériodicité. On peut seulement supposer que $P_{6n}=\emptyset$, puisqu'on pourra toujours utiliser \ref{dist}.

\subsection{La conjecture de Lindenstrauss-Tsukamoto}

On a vu en introduction qu'on cherchait quand un système (X,T) pouvait se plonger dans un décalage $((|0,1]^d)^\mathbb{Z},\sigma)$, d étant le plus petit possible. On cherche en quelque sorte à étendre le théorème de Jaworski.

La moyenne dimension va nous fournir un outil pour distinguer entre des systèmes qui seront tous de dimension infinie. En revanche, un critère de la forme $P_n=\emptyset$ ne semble pas satisfaisant. En effet, la présence de points périodiques, quand elle n'est pas trop importante, peut ne pas créer d'obstruction à un plongement. Si on reprend l'exemple précédent de la sphère $(\mathbb{S}^2,\tau)$, malgré la présence de points fixes, on a un plongement de $(\mathbb{S}^2,\tau)$ dans le décalage $(([0,1]^2)^\mathbb{Z},\sigma)$ :
 \begin{align*}
 \mathbb{S}^2 &\longrightarrow (\mathbb{R}^2)^\mathbb{Z}\\
 (x_1,x_2,x_3)&\longmapsto (...,(x_1+x_3,x_2),(x_1-x_3,x_2),(x_1+x_3,x_2),(x_1-x_3,x_2),...)
 \end{align*}

Cela motive l'introduction de la dimension périodique.

\begin{Def}
On définit la dimension périodique de (X,T) par :
\[perdim(X,T)=\sup_{k\in \mathbb{N}}\frac{dim(P_k)}{k}\]
\end{Def}

\begin{Ex}
Un système apériodique est de dimension périodique nulle. En revanche, un système peut être de dimension périodique nulle et admettre des points périodiques.
\end{Ex}
\begin{Ex}
On a $perdim((|0,1]^d)^\mathbb{Z},\sigma)=d$. En effet, si $k\in \mathbb{N}$, en notant $H_i$ l'ensemble des points périodiques de période i, on a :
\[dim(P_k)=dim(H_1\cup H_2\cup ... H_k)=max_{0\leq i \leq k}(dim(H_i))=kd\]
\end{Ex}

\begin{propo}\label{sousperdim}
Si $\tilde{X}$ est un sous-ensemble fermé de X, invariant par T, alors si $\tilde{T}$ désigne la restriction de T à $\tilde{X}$, on a $perdim(\tilde{X},\tilde{T})\leq perdim(X,T)$.
\end{propo}
\begin{proof}
Soit $k\in \mathbb{N}$. Si x est un point périodique de période inférieure à k pour $(\tilde{X},\tilde{T})$, il l'est aussi pour $(X,T)$. Donc $perdim(\tilde{X},\tilde{T})\leq perdim(X,T)$.
\end{proof}

En particulier, une condition nécessaire pour que le système dynamique (X,T) se plonge dans $((|0,1]^d)^\mathbb{Z},\sigma)$ est : $perdim(X,T)\leq d$.

\begin{propo}
Soit $m\in \mathbb{N}$. On a $perdim(X,T^m)\leq m perdim(X,T)$
\end{propo}
\begin{proof}
Soit $k\in \mathbb{N}$. Alors si x est un point périodique de période k pour $T^m$, c'est un point périodique de période mk pour T. On a donc $dim(P_k)\leq m\times \sup_{i\geq 0}\frac{dim(P_i)}{i}$ . D'où finalement $perdim(X,T^m)\leq m perdim(X,T)$.
\end{proof}

La conjecture de Lindenstrauss-Tsukamoto s'énonce ainsi : Si $mdim(X,T)<\frac{d}{2}$ et $perdim(X,T)<\frac{d}{2}$, alors le système dynamique (X,T) se plonge dans le décalage $(([0,1]^d)^\mathbb{Z},\sigma)$.

Le théorème de Jaworski entre dans le cadre de la conjecture, avec d=1, car la moyenne dimension et la dimension périodique sont tous les deux nuls. Dans la suite, on présente différents résultats qui sont tous des cas particuliers de la conjecture. On utilise à chaque fois la méthode décrite dans la section 3.1. La difficulté est toujours de montrer la densité des espaces considérés : chaque démonstration est précédée d'un lemme d'approximation. Quelques lemmes techniques, nécessaires à leur démonstration, ont été placés dans la section finale.

\subsection{Plongement de l'ensemble des points périodiques}

On considère ici une première méthode : dans \cite{YG}, on parvient à étendre la démonstration de Jaworski au cas où $perdim(X,T)<\frac{d}{2}$, en plongeant l'ensemble des points périodique dans le décalage $(([0,1]^d)^\mathbb{N},\sigma)$.

On rappelle la définition suivante :

\begin{Def}
Soit $\{ v_1, v_2,..., v_r\}$ une famille de vecteurs de $\mathbb{R}^m$. Alors on dit qu'elle est affinement indépendante si $\{v_2-v_1,v_3-v_1,...,v_{r}-v_1\}$ est une famille libre. Cela équivaut à ce que si $\sum_{i=1}^r \lambda_i v_i=0$ et $\sum_{i=1}^r\lambda_i=0$, alors tous les coefficients $\lambda_i$ sont nuls.
\end{Def}

\begin{Notation}
Soient $n_1\geq n_2$ deux entiers et v un vecteur de taille $n_2$. Alors on note $v^{*n_1}$ le vecteur de taille $n_1$ défini par :
\[ \forall 1\leq k\leq n_1, v^{*n_1}|_k=v|_{k \:mod\: n_2} \]
\end{Notation}

\begin{Lemme}\label{approx2}
Soient $n_1, n_2 \in \mathbb{N}$ tels que $n_1\geq n_2$, et $R_1, R_2 \subset X$ des ensembles fermés disjoints. Pour $i=1,2$, soit $\alpha_i$ un recouvrement ouvert de $R_i$ tel que $ord(\alpha_i)<\frac{dn_i}{2}$.
Pour tout $U\in \alpha_i$, on se donne un point $q_U\in U$ et un $n_i$-uplet $v_U \in ([0,1]^d)^{n_i}$. Soit $\epsilon >0$.
Alors on peut trouver deux fonctions continues $F_i:X\to ([0,1]^d)^{n_i}$ pour $i=1,2$ telles que :
\begin{enumerate}
\item $\forall U\in \alpha_i ,\Arrowvert F_i(q_U)-v_U \Arrowvert _\infty <\epsilon$
\item $\forall x\in R_i, F(x) \in Conv(\{ F_i(q_U) | x\in U \}$
\item si $x_1\in R_1$ et $x_2\in R_2$, $F_1(x_1)\neq F_2(x_2)^{*n_1}$
\end{enumerate}
\end{Lemme}

\begin{proof}
On va définir les fonctions $F_i$ ainsi :
\[ \forall x\in X, F_i(x) =\sum_{U\in \alpha_i}\rho_U(x)\vec{v_U} \]
où $\{ \rho_U\}_{U\in \alpha_i}$ est une partition de l'unité subordonnée à $\alpha_i$ telle que pour tout $U\in \alpha_i$, $\rho_U(q_U)=1$.
Pour $U\in \alpha_2$, $\vec{v_U}=v_U$ et pour $U\in \alpha_1$, on définira $\vec{v_U}$ par la suite. La propriété 2. sera alors immédiatement vérifiée.

Quant à la propriété 3., elle s'écrit :
\[ \text{si }x\in R_1 \text{ et } y\in R_2, \sum_{U\in \alpha_1^x}\rho_U(x)\vec{v_U}- \sum_{U\in \alpha_2^y}\rho_U(y)\vec{v_U}^{*n_1}\neq 0 \]
où on a posé :
\[ \alpha_i^x=\{ U\in \alpha_i \: | \: \rho_U(x)>0 \} \]

On a affaire à une combinaison d'au plus $ord(\alpha_1) + ord(\alpha_2)+2$ vecteurs. Or $ord(\alpha_1)\leq \frac{n_1d}{2}-\frac{1}{2}$ et $ord(\alpha_2)\leq \frac{n_2d}{2}-\frac{1}{2}\leq \frac{n_1d}{2}-\frac{1}{2}$. Donc il s'agit d'au plus $n_1d+1$ vecteurs.

Fixons dans un premier temps $(x,y)\in (R_1,R_2)$. Soit E le sous-espace de $([0,1]^d)^{n_1}$ engendré par la famille $\{ \vec{v_U}^{*n_1}\: | \: U\in \alpha_2^y \}$. Alors $dim(E)\leq card (\alpha_2^y) \leq \frac{n_2d}{2}+\frac{1}{2}$. On considère deux cas.

Supposons que $dim(E)=\frac{n_2d}{2}+\frac{1}{2}$. Alors la famille $\{ \vec{v_U}^{*n_1}\: | \: U\in \alpha_2^y \}$ est libre. Donc d'après \ref{indep} (cas 2) dans la section 3.6, pour presque tout choix de $\{ \vec{v_U}\}_{U\in \alpha_1^x}$ dans $([0,1]^{dn_1})^{card(\alpha_1^x})$, les vecteurs de $\{ \vec{v_U}^{*n_1}\: | \: U\in \alpha_2^y \} \cup \{ \vec{v_U}\}_{U\in \alpha_1^x}$ sont affinement indépendants. Supposons alors :
\[ \sum_{U\in \alpha_1^x}\rho_U(x)\vec{v_U}- \sum_{U\in \alpha_2^y}\rho_U(y)(\vec{v_U}|_{k \: mod\: n_2})_{k=1}^{n_1}= 0 \]
Comme les sommes $\sum_{U\in \alpha_1^x}\rho_U(x)$ et $\sum_{U\in \alpha_1^y}\rho_U(y)$ sont égales à 1, la somme des coefficients en jeu est nulle, or ces coefficients sont tous différents de 0, c'est une contradiction.

Dans l'autre cas, si $dim(E)<\frac{n_2d}{2}+\frac{1}{2}$, alors d'après le lemme précédent, pour presque tout tout choix de $\{ \vec{v_U}\}_{U\in \alpha_1^x}$ dans $([0,1]^{dn_1})^{card(\alpha_1^x})$, les espaces $Vect(\{ \vec{v_U}^{*n_1}\: | \: U\in \alpha_2^y \})$ et $ Vect(\{ \vec{v_U}\}_{U\in \alpha_1^x})$ sont complémentaires, et la famille $\{ \vec{v_U}\}_{U\in \alpha_1^x}$ est libre. Supposons alors :
\[ \sum_{U\in \alpha_1^x}\rho_U(x)\vec{v_U}= \sum_{U\in \alpha_2^y}\rho_U(y)\vec{v_U}^{*n_1} \]
Comme $Vect(\{ \vec{v_U}^{*n_1}\: | \: U\in \alpha_2^y \})$ et $ Vect(\{ \vec{v_U}\}_{U\in \alpha_1^x})$ sont complémentaires, on obtient :
\[ \sum_{U\in \alpha_1^x}\rho_U(x)\vec{v_U}=0\]
Comme $\{ \vec{v_U}\}_{U\in \alpha_1^x}$ est libre, tous les coefficients en jeu dans cette combinaison linéaire sont nuls, ce qui est une contradiction.

Finalement, comme il existe un nombre fini de familles de la forme $\alpha_1^x$ et $\alpha_2^y$, on peut choisir la famille $\{ \vec{v_U}\}_{U\in \alpha_1^x}$ dans un ensemble de complémentaire négligeable dans $([0,1]^{dn_1})^{card(\alpha_1^x)}$, donc de façon à vérifier la propriété 1.
\end{proof}

\begin{Lemme}\label{approx3}
Soient $n,l\in \mathbb{N}$ tels que $1\leq l\leq n-1$. Soit R un sous-ensemble fermé de X tel que $\forall i\in \{ 1,2,...,n-1\}, R\cap T^iR=\emptyset$. Soit $\alpha_i$ un recouvrement ouvert de R tel que $ord(\alpha)<\frac{dn}{2}$.Soit $\epsilon >0$. Pour tout $U\in \alpha$, on se donne un point $q_U\in U$ et un N-uplet $v_U \in ([0,1]^d)^n$.
Alors on peut trouver une fonction continue $F:R\to ([0,1]^d)^n$ telle que :
\begin{enumerate}
\item $\forall U\in \alpha ,\Arrowvert F(q_U)-v_U \Arrowvert _\infty< \epsilon$
\item $\forall x\in R, F(x) \in Conv(\{ F(q_U) | x\in U \}$
\item si $x,y \in R$, $F(x)\neq (F(y))^{\bullet l}$
où pour un vecteur v de taille $n>l$, le vecteur $v^{\bullet l}$ de taille n est définit par : $v^{\bullet l}|_k=v_{k+l\: mod\: n}$
\end{enumerate}
\end{Lemme}

\begin{proof}

On définit la fonction F ainsi :
\[ \forall x\in R, F(x) =\sum_{U\in \alpha_i}\rho_U(x)\vec{v_U} \]
où $\{ \rho_U\}_{U\in \alpha_i}$ est une partition de l'unité subordonnée à $\alpha$ telle que pour tout $U\in \alpha_i$, $\rho_U(q_U)=1$.
On définira $\{ \vec{v_U}\}_{U\in \alpha}$ par la suite. La propriété 2. sera alors immédiatement vérifiée.

Pour $x\in R$, on pose : $\alpha_x =\{ U\in \alpha \: | \: \rho_U(x)>0\}$.
Alors la propriété 3 pour x s'écrit :

\[  \forall i\in \{ 0,1,...,nd-1\}, \sum_{U\in \alpha_x}\rho_U(x)\vec{v_U}|_i \neq \sum_{U\in \alpha_y}\rho_U(y)\vec{v_U}|_{i+dl\: mod\: nd} \]

Pour tout U dans $\alpha$, on définit un vecteur d'entiers $V_U=(v_U^1, v_U^2, ..., v_U^{nd})\in \mathbb{N}^{nd}$, de manière à ce que les indices $v_U^i$ soient tous distincts (U décrivant $\alpha$ et $i$ décrivant $\{1,...,nd\}$. Soit alors M la matrice dont les colonnes sont les vecteurs $V_U$ pour $U\in \alpha_x$ et $V_U^{\bullet dl}$ pour $U\in \alpha_y$. M a nd lignes et au plus nd+1 colonnes car $D(ord(\alpha)+1)\leq nd+1$. Comme $1\leq l\leq n-1$, il n'y a pas de coefficients égaux sur une même ligne ou une même colonne.

Supposons que le nombre de colonnes, k,  est inférieur ou égal à $nd$. Alors, en ne gardant que les k premières lignes de M, on obtient une matrice carrée à laquelle on peut appliquer \ref{matrice}. Donc pour presque tout choix de vecteurs $(v_U)_{U\in \alpha}$, la famille $\{ (v_U)_{U\in \alpha_x}\cup (v_U^{\bullet l})_{U\in \alpha_y}\}$ est une famille libre.

Supposons maintenant que $k=nd+1$. $nd+1=2(ord(\alpha)+1)$ donc nd+1 est pair. En appliquant \ref{matrice2} de la section 3.6. à $i=\frac{nd+1}{2}$, on obtient que pour presque tout choix de vecteurs $(v_U)_{U\in \alpha}$, la vecteurs de $\{ (v_U)_{U\in \alpha_x}\cup (v_U^{\bullet l})_{U\in \alpha_y}\}$ sont affinement indépendants.

Finalement, dans les deux cas, on peut conclure comme dans \ref{approx2}.

\end{proof}

On rappelle qu'on note P l'ensemble des points périodiques, et pour $n\in \mathbb{N}$, $H_n$ l'ensemble des points périodiques de période n.

\begin{thm}
Supposons que $perdim(X,T)<\frac{d}{2}$. Alors l'espace $D_{(P\times P)\backslash \Delta }$ est dense dans $C(X,[0,1]^d)$.
\end{thm}

\begin{proof}
On va utiliser la méthode décrite dans la section 3.1. On a :

\[ (P\times P)\backslash \Delta = (\bigcup_{n\in \mathbb{N}} H_n\times H_n)\backslash \Delta \cup (\bigcup_{n\neq m} H_n\times H_m)\]

$(\bigcup_{n\in \mathbb{N}} H_n\times H_n)\backslash \Delta$ et $(\bigcup_{n\neq m} H_n\times H_m)$ sont deux espaces de Lindelöf.

Soit $n\in \mathbb{N}$ et $x_n\in H_n$. On peut trouver un voisinage $U_x$ de $x_n$ dans $H_n$ tel que $\bar{U_x} \subset H_n$ et $\forall i \in \{ 1,2,...,n-1\}, \bar{U_x}\cap T^i\bar{U_x}=\emptyset$. On peut donc recouvrir l'espace $(\bigcup_{n\neq m} H_n\times H_m)$ par des compacts de la forme $K=\bar{U}_{x_n}\times  \bar{U}_{y_m}$.

Soit maintenant $(x,y)\in (H_n\times H_n)\backslash \Delta$. Alors on définit un voisinage compact $K$ de la forme $\bar{V}_x\times \bar{V}_y$, où $V_x$ et $V_y$ sont des voisinages respectifs de x et de y dans $H_n$, vérifiant :
\begin{enumerate}
\item $\bar{V}_x \subset H_n$ et $\bar{V}_y \subset H_n$
\item $\bar{V}_x\times \bar{V}_y \subset (P\times P)\backslash \Delta$
\item si $S, Tx,..., T^{n-1}x,y,Ty,...,T^{n-1}y$ sont deux à deux distincts, $V_x$ et $V_y$ doivent être choisis tels que $\bar{V}_x,T\bar{V}_x,.., T^{n-1}\bar{V}_x,\bar{V}_y, T\bar{V}_y,...,T^{n-1}\bar{V}_y$ sont deux à deux disjoints.
\item si en revanche il existe $1\leq k\leq n-1$ tel que $y=T^kx$, on choisit $V_x$ tel que $\forall i\in \{1, 2,...,n-1\}, \bar{V_x}\cap T^i\bar{V_x}=\emptyset$ et on pose $V_y=T^kV_x$.
\end{enumerate}
On obtient un recouvrement de $(\bigcup_{n\in \mathbb{N}} H_n\times H_n)\backslash \Delta$ par des compacts de la forme $K=\bar{V}_x\times \bar{V}_y$.

Il reste à montrer que tous les ensembles $D_K$ considérés sont denses dans $C(X,[0,1]^d)$. Supposons que K s'écrive $R_1\times R_2$ où $\bar{R_1} \subset H_{n_1}$ et $\bar{R_2} \subset H_{n_2}$.

Soient $\tilde{f}\in C(X,[0,1]^d)$ et $\epsilon>0$. Comme $perdim(X,T)<\frac{d}{2}$, on a $\frac{dim(\bar{R_1})}{n_1}<\frac{d}{2}$ et $\frac{dim(\bar{R_2})}{n_2}<\frac{d}{2}$. Pour $i=1,2$, on peut donc choisir $\alpha_i$ un recouvrement ouvert de $R_i$ tel que $ord(\alpha_i)<\frac{dn_i}{2}$, et $max_{U\in \alpha_i; 0\leq k \leq n_i-1}diam(\tilde{f}(T^kU))< \frac{\epsilon}{2}$.
Pour tout $U\in \alpha_i$, on se donne un point $q_U\in U$ et on définit $v_U=(\tilde{f}(T^kq_U))_{k=0}^{n_i-1}$.

On distingue alors deux cas.

Dans le premier cas, on suppose que $\bar{R_1}, T\bar{R_1}, ..., T^{n_1-1}\bar{R_1}, \bar{R_2}, T\bar{R_2}, ..., T^{n_2-1}\bar{R_2}$ sont deux à deux disjoints. Cela correspond à la propriété 3 dans le choix d'un voisinage de $(x,y)\in (H_n\times H_n)\backslash \Delta$ et au cas où K est un voisinage de $(x,y)\in (\bigcup_{n\neq m} H_n\times H_m)$ (en effet, $\bigcup_{i=1}^{n_1}T^{i}\bar{R_1} \subset H_{n_1}$, $\bigcup_{i=1}^{n_2}T^{i}\bar{R_2} \subset H_{n_2}$ et $H_{n_1}$ et $H_{n_2}$ sont disjoints ). Alors en appliquant \ref{approx2} à $R_1$ et $R_2$, on obtient deux fonctions continues $F_1$ et $F_2$.

On peut alors définir une fonction $f'$ sur $Z=\bigcup_{k=0}^{n_1-1}T^k\bar{R_1}\cup \bigcup_{k=0}^{n_2-1}T^k\bar{R_2}$ en posant :
\[  \forall z\in \bar{R_i}, \forall 0\leq k\leq n_i-1, f'(T^kz)=F_i(z)|_k   \]
D'après les propriétés 1 et 2 dans \ref{approx2}, on obtient $\Arrowvert \tilde{f}|_Z-f'\Arrowvert < \epsilon$. D'après le théorème d'extension de Tietze, on peut prolonger f' en une fonction $f\in C(X,[0,1]^d)$ telle que
$\Arrowvert \tilde{f}-f\Arrowvert < \epsilon$. Alors $f$ appartient à $D_K$ . En effet, supposons qu'il existe $(x,y)\in K$ tel que $\forall a\in \mathbb{Z}, f(T^ax)=f(T^ay)$. Alors en particulier $(f(x),..., f(T^{n_1-1}x))=(f(y),..., f(T^{n_1-1}y))$. On remarque que pour $n_1> i\geq n_2, f(T^iy)=  F_2(y)|_{i \: mod\: n_2}$. En effet,si on pose $k=i$ mod $n_2$, comme $y\in \bar{R_2}\subset H_2$, on a $T^iy=T^ky$. Donc $F_1(x)=(F_2(y)|_{k \: mod\: m})_{k=1}^{n_1}$, ce qui est impossible.

Dans l'autre cas, K est un voisinage de $(x,y)\in (H_n\times H_n)\backslash \Delta$ vérifiant la propriété 4. On pose $n=n_1=n_2$. $\bar{R_1}, T\bar{R_1}, ..., T^{n_1-1}\bar{R_1}$ sont deux à deux disjoints mais il existe $1\leq l\leq n-1$ tel que $\bar{R_2}=T^l\bar{R_1}$. On peut appliquer \ref{approx3} à $R_1$. On obtient une fonction continue F. On procède de même que dans le premier cas : on définit une fonction $f'$ sur $Z=\bigcup_{k=0}^{n-1}T^k\bar{R_1}$ par :
\[  \forall z\in \bar{R_1}, \forall 0\leq k\leq n-1, f'(T^kz)=F(z)|_k   \]
Et on prolonge cette fonction $f'$ en une fonction $f\in C(X,[0,1]^d)$ telle que $\Arrowvert \tilde{f}-f\Arrowvert < \epsilon$. Supposons qu'il existe $(x,y)\in K$ tel que $\forall a\in \mathbb{Z}, f(T^ax)=f(T^ay)$. Alors en particulier $(f(x),..., f(T^{n-1}x))=(f(y),..., f(T^{n-1}y))$, donc $F(x)= (F(y))^{\bullet l}$ , ce qui est impossible.

\end{proof}

Ici on parvient ainsi à démontrer la conjecture de Lindenstrauss-Tsukamoto dans le cas où X est de dimension finie.

\begin{Lemme}\label{approx4}
Soient N et n dans $\mathbb{N}$. Soit $\alpha$ un recouvrement ouvert de X tel que $ord(\alpha)+1\leq (N-1-n)\frac{d}{2}$.Soit $\epsilon >0$. Pour tout $U\in \alpha$, on se donne un point $q_U\in U$ et un N-uplet $v_U \in ([0,1]^d)^N$.
Alors on peut trouver une fonction continue $F:X\to ([0,1]^d)^N$ telle que :
\begin{enumerate}
\item $\forall U\in \alpha ,\Arrowvert F(q_U)-v_U \Arrowvert _\infty< \epsilon$
\item $\forall x\in X, F(x) \in Conv(\{ F_i(q_U) | x\in U \}$
\item si $x\in X$, $F(x)|_1^{N-1}\notin V_{N-1}^n$, où
\[ V_{N-1}^n=\{z\in ([0,1]^d)^{N-1}\: | \: \forall 0 \leq a,b\leq N, (a=b\: mod\: n) \Rightarrow z_a=z_b\}\]

\end{enumerate}
\end{Lemme}

\begin{proof}
On définit la fonction F ainsi :
\[ \forall x\in X, F(x) =\sum_{U\in \alpha}\rho_U(x)\vec{v_U} \]
où $\{ \rho_U\}_{U\in \alpha}$ est une partition de l'unité subordonnée à $\alpha$ telle que pour tout $U\in \alpha$, $\rho_U(q_U)=1$.
On définira $\{ \vec{v_U}\}_{U\in \alpha}$ par la suite. La propriété 2. sera alors immédiatement vérifiée.

Soit $x\in X$. On pose : $\alpha_x =\{ U\in \alpha \: | \: \rho_U(x)>0\}$.
Alors la propriété 3 pour x s'écrit :

\[  \sum_{U\in \alpha_x}(\rho_U(x)\vec{v_U}|_1^{N-1})\notin V_{N-1}^n \]

Il s'agit d'une combinaison linéaire d'au plus $ord(\alpha)+1$ vecteurs. On va utiliser le cas 1. dans \ref{indep} de la section 3.6., avec $r\leq ord(\alpha)+1$ et $s=nd$ car $dim(V_{N-1}^n)=nd$. Alors comme $ord(\alpha)+1+nd\leq (N-1)d$, pour presque tout choix de $\{ \vec{v_U}\}_{U\in \alpha}$ dans $(([0,1]^d)^{N-1})^{ord(\alpha)+1}$, les espaces $Vect(\{ \vec{v_U}\}_{U\in \alpha})$ et $V_{N-1}^n$ sont complémentaires. Comme ses coefficients sont non nuls, la combinaison $\sum_{U\in \alpha_x}(\rho_U(x)\vec{v_U}|_1^{N-1})$ est non nulle donc ne peut pas être dans $V_{N-1}^n$.

Comme il y a un nombre fini d'ensembles de la forme $\alpha_x$, on en déduit qu'on peut choisir la famille $\{ \vec{v_U}\}_{U\in \alpha}$ de manière à ce qu'elle vérifie la propriété 1.

\end{proof}

\begin{thm}
Supposons que X est de dimension finie n, et $perdim(X,T)<\frac{d}{2}$. Alors le système dynamique (X,T) se plonge dans le décalage $(([0,1]^d)^\mathbb{Z},\sigma)$.
\end{thm}

\begin{proof}
Posons $\Gamma = (X \times X)\backslash (\Delta \cup (P\times P))$. On peut écrit $\Gamma=D_1 \cup D_2$ où $D_1= (P\times X\backslash P)\cup (X\backslash P\times P)$ et $D_2=\Gamma \backslash D_1$.

Soit $(x,y)\in D_1$. Par exemple on peut supposer que x n'est pas un point périodique et que $y\in H_m$ où $n\in \mathbb{N}$. On choisit $S\in \mathbb{N}$ tel que $n+1\leq (S-\frac{m}{2})d$. On peut trouver un voisinage ouvert de x tel que $U\subset X\backslash P$ et tel que pour $l\in \{0,1...,2S\}$, $\bar{U} \cap T(\bar{U})=\emptyset$. Enfin on choisit un voisinage W de y tel que $W \subset \bar{W} \subset H_m$, et on définit $K=\bar{U}\times \bar{X}$.

Montrons que $D_K$ est dense dans $C(X,[0,1]^d)$. Soit $f\in C(X,[0,1]^d)$ et $\epsilon >0$. $\bar{U}$ et $\bar{W}$ sont des compacts disjoints donc $dist(\bar{U},\bar{X})>0$. On peut supposer sans perte de généralité que $\epsilon < dist(\bar{U},\bar{X})$. Soit $\alpha$ un recouvrement ouvert fini de $\bar{U}$ tel que $ord(\alpha)\leq n, \max_{V\in \alpha, k \in {0,1,...,2S}} diam(f(T^k(V)))<\frac{\epsilon}{2}$, et $\max_{V\in \alpha}diam(V)<\epsilon$. Pour tout $V\in \alpha$ on choisit un point $q_V \in V$ et on note : $\tilde{v}_V=(f(T^k(q_V)))_{k\in \{ 0,1,...,2S\} }$. On peut alors appliquer \ref{approx4} aux entiers 2S+1 et m. On obtient une fonction continue $F:X \to ([0,1]^d)^{2S+1}$.

Posons $Z=\bigcup_{k=0}^{2S}T^k(\bar{U})$. On définit $g:Z\to [0,1]^d$ par :
\[ \forall k\in \{ 0,1,...2S\},\forall z\in \bar{U}, g_(T^k(z))=F(z)_k \]
Alors comme dans le théorème précédent, on peut prolonger g en une fonction $\tilde{f}$ telle que $\Arrowvert f-\tilde{f} \Arrowvert _\infty < \epsilon$. Alors $\tilde{f}$ est dans $D_K$. En effet, supposons qu'il existe $(x,y)\in K$ tel que $\forall a\in \mathbb{Z}, f(T^ax)=f(T^ay)$. Alors en particulier $(f(Tx),..., f(T^{2S}x))=(f(Ty),..., f(T^{2S}y))$. Comme $(f(Ty),..., f(T^{2S}y))\in V_{2S}^m$, on en déduit que $F(x)|_1^{2S}\in V_{2S}^m$, ce qui est impossible. Finalement, on a prouvé que
$D_K$ est dense dans $C(X,[0,1]^d)$.

On va conduire un raisonnement analogue dans le deuxième cas : soit $(x,y)\in D_2$. Alors d'après \ref{dist}, on peut trouver des indices $i_0,i_1,...,i_{2n}$ tels que les points $T^{i_0}(x),...,T^{i_{2n}}(x),T^{i_0}(y),...,T^{i_{2n}}(y)$ sont tous distincts. Alors on peut trouver deux voisinages compacts respectifs de $x$ et $y$, $U_1$ et $U_2$, tels que $T^{i_0}(U_1),...,T^{i_{2n}}(U_1),T^{i_0}(U_2),...,T^{i_{2n}}(U_2)$ sont deux à deux disjoints. On pose alors $K=U_1\times U_2$.

Montrons que $D_K$ est dense dans $C(X,[0,1]^d)$. Soit $f\in C(X,[0,1]^d)$ et $\epsilon >0$.On peut supposer sans perte de généralité que $\epsilon < dist(\bar{U_1},\bar{U_2})$. Pour $i=1,2$, soit $\alpha_i$ un recouvrement ouvert de $U_i$ tel que $ord(\alpha_i)<n$,$ \max_{V\in \alpha_i, k \in {0,1,...,2n}} diam(f(T^k(V)))<\epsilon$, et $\max_{V\in \alpha_i}diam(V)<\epsilon$.
Pour tout $U\in \alpha_i$, on se donne un point $q_U\in U$ et on pose $v_U=(f(T^{i_k}q_U))_{k=0}^{2n}$. On peut alors appliquer \ref{approx2} en posant $n_1=n_2=n$. On obtient deux fonctions continues $F_1$ et $F_2$.

Posons $Z=\bigcup_{k=0}^{2n}T^k(\bar{U_1}\cup \bar{U_2})$. On définit $g:Z\to [0,1]^d$ par :
\[ \forall k\in \{ 0,1,...2S\},\forall z\in \bar{U_i}, g(T^k(z))=F_i(z)_k \]
Alors comme précédemment, on peut prolonger g en une fonction $\tilde{f}$ telle que $\Arrowvert f-\tilde{f} \Arrowvert _\infty < \epsilon$. Alors $\tilde{f}$ est dans $D_K$. En effet, supposons qu'il existe $(x,y)\in K$ tel que $\forall a\in \mathbb{Z}, f(T^ax)=f(T^ay)$. Alors en particulier $(f(T^{i_0}x),..., f(T^{i_{2n}}x))=(f(T^{i_0}y),..., f(T^{i_{2n}}y))$, c'est-à-dire $F_1(x)=F_2(y)$. C'est une contradiction. Finalement, on a prouvé que $D_K$ est dense dans $C(X,[0,1]^d)$.
\end{proof}

On considère maintenant  un système dynamique : $(X,T)=(\Pi_{i\in \mathbb{N}}X_i,\Pi_{i\in \mathbb{N}}T_i)$, produit dénombrable de systèmes dynamiques $(X_i,T_i)$, où $X_i$ est un espace compact métrisable de dimension finie et $T_i$ un homéomorphisme sur $X_i$.

On notera $X^{(n)}$ le produit fini $X_1\times X_2\times ... \times... X_n$ et $T^{(n)}$ le produit fini $T_1 \times T_2 \times ... T_n$.

\begin{propo}
Supposons que $\liminf_{n\to \inf} perdim(X^{(n)},T^{(n))})<\frac{d}{2}$. Alors le système dynamique (X,T) se plonge dans le décalage $(([0,1]^d)^\mathbb{Z},\sigma)$.
\end{propo}

\begin{proof}
Soient $x^{(0)}$ et $y^{(0)}$ distincts dans X. Il existe $k\in \mathbb{N}$ tel que $x^{(0)}_k\neq y^{(0)}_k$. On peut trouver deux compacts U et V dans X, voisinages respectifs de $x^{(0)}$ et $y^{(0)}$, tels que si $(x,y)\in U\times V$, alors $x_k\neq y_k$.
Posons $K=U\times V$.

On va utiliser la méthode décrite dans la section 3.1. Il reste donc à montrer que $D_K$ est dense dans $C(X,[0,1]^d)$.

Soit $f\in C(X,[0,1]^d)$, et $\epsilon >0$. f est uniformément continue sur X. Il existe donc $\eta>0$ tel que si $d(x,y)<\eta$ alors $\parallel f(x)-f(y)\parallel <\frac{\epsilon}{2}$.
Soit n un entier supérieur à k, tel que si x et y ont leurs n premières coordonnées égales, alors $d(x,y)<\eta$. Soit $z\in X$.
Posons :
\begin{align*}
h : &X \longrightarrow [0,1]^d\\
&x \longmapsto f(x_1,x_2,...,x_n,z_{n+1},z_{n+2}...)
\end{align*}
Alors $h\in C(X,[0,1]^d)$, $\parallel f-h\parallel <\frac{\epsilon}{2}$, et h(x) ne dépend que des n premières coordonnées de x.
On peut donc définir :
\begin{align*}
\tilde{h}:&X^{(n)}\longrightarrow [0,1]^d\\
&(x_1,...,x_n)\longmapsto h(x_1,...,x_n,z_{n+1},...)
\end{align*}

Comme $\liminf_{n\to +\inf} perdim(X^{(n)},T^{(n))})<\frac{d}{2}$, on peut supposer, quitte à prendre un n plus grand, que $perdim(X^{(n)},T^{(n))})<\frac{d}{2}$.
Alors comme $X^{(n)}$ est de dimension finie, d'après le théorème précédent, il existe une fonction $\tilde{f} \in C(X^{(n)},[0,1]^d)$ telle que $\parallel \tilde{f}-\tilde{h}\parallel <\frac{\epsilon}{2}$ et telle que si x et y sont deux éléments distincts de $X^{(n)}$, alors il existe $i\in \mathbb{N}$ tel que :
$\tilde{f}(T^i(x))\neq \tilde{f}(T^i(y))$.
On pose alors :
\begin{align*}
f':&X\longrightarrow [0,1]^d\\
&x\longmapsto \tilde{f}(x_1,...,x_n)
\end{align*}

Alors $\parallel f'-h\parallel <\frac{\epsilon}{2}$, donc $\parallel f'-f\parallel <\epsilon$, et $f'\in D_K$.

On en déduit que que $D_K$ est dense dans $C(X,[0,1]^d)$, puis que finalement $D_{(X\times X) \backslash \Delta}$ est dense dans $C(X,[0,1]^d)$.
\end{proof}

\begin{Ex}

\end{Ex}

La condition $\liminf_{n\to \inf} perdim(X^{(n)},T^{(n))})<\frac{d}{2}$ est strictement plus contraignante que $perdim(X,T)<\frac{d}{2}$. En effet, pour tout $n\in \mathbb{N}$, d'après \ref{sousperdim}, $perdim(X,T)\leq perdim(X^{(n)},T^{(n))})$. Donc si $\liminf_{n\to \inf} perdim(X^{(n)},T^{(n))})<\frac{d}{2}$, on a $perdim(X,T)<\frac{d}{2}$.

D'autre part, considérons le système dynamique $(X,T)=(\Pi_{i\in \mathbb{N}}X_i,\Pi_{i\in \mathbb{N}}T_i)$, où pour tout $i\in \mathbb{N}$, $X_i$ est le tore de dimension 1, et $T_i$ est l'application identité, sauf si i est une puissance de 2 : alors $T_i$ est la rotation d'angle $\frac{2\pi}{i}$. Alors $perdim(X,T)=0$. Cependant, pour tout $n\in \mathbb{N}$, $1\leq perdim(X^{(n)},T^{(n))})\leq 2$.

\subsection{n-marqueurs}

\subsubsection{n-marqueurs}

\begin{Def}
Soit $n\in \mathbb{N}$. Soit F un sous-ensemble de X. On dit que c'est un n-marqueur de (X,T) si :
\begin{enumerate}
\item $\forall i\in \{ 1,2,...,n-1\} , F\cap T^i(F)=\emptyset$
\item $\exists m\in \mathbb{N}, X=\bigcup_{i=1}^mT^i(F)$
\end{enumerate}
On dit que (X,T) a la propriété de marqueur si pour tout $n\in \mathbb{N}$ il existe un n-marqueur ouvert de (X,T).
\end{Def}

\begin{Def}
(X,T) a la propriété topologique forte de Rokhlin si pour tout $n\in \mathbb{N}$ il existe une fonction continue $f:X\to \mathbb{R}$ telle qu'en définissant :
\[ E_f=\{ x\in X | f(T(x))\neq f(x)+1\} \]
on ait :
\[ \forall i\in \{ 1,2,...,n-1\}, E_f\cap T^i(E_f)=\emptyset \]

\end{Def}

\begin{thm}\label{Rok}
(X,T) a la propriété de marqueur si et seulement si (X,T) a la propriété topologique forte de Rokhlin.
\end{thm}

\begin{proof}
Supposons que (X,T) a la propriété topologique forte de Rokhlin. Soit $n\in \mathbb{N}$. Soit $f:X\to \mathbb{R}$ une fonction continue telle que :
\[ \forall i\in \{ 1,2,...,n-1\}, E_f\cap T^i(E_f)=\emptyset \]
Alors $X=\bigcup_{i=1}^\infty T^i(E_f)$. En effet, supposons qu'il existe $y\in X \backslash \bigcup_{i=1}^\infty T^i(E_f)$. Alors la suite $(f(T^{-i}(y)))_{i\in \mathbb{N}}$ diverge vers $-\infty$. C'est impossible car la fonction f étant continue sur le compact X, elle est bornée. Finalement, par compacité de X, il existe $m\in \mathbb{N}$ tel que $X=\bigcup_{i=1}^m T^i(E_f)$.
Comme $E_f$ est ouvert, finalement $E_f$ est un n-marqueur. Donc (X,T) a la propriété de marqueur.\\

Réciproquement, supposons que (X,T) a la propriété de marqueur. Soit $n\in \mathbb{N}$ et F un n-marqueur. Soit U un ouvert de X, contenant F, tel que $\forall i\in \{1,2,...,n-1\}, U\cap T^iU=\emptyset$.

Soit $\rho :X\to [0,1]$ une fonction continue qui vaut 1 en tout point de F, et dont le support est inclus dans U.
On définit une marche aléatoire sur X comme suit : si à l'instant t on se trouve sur le point y, à l'instant t+1, on aura mis fin à la marche avec probabilité $\rho(y)$ et on se place en $T^{-1}(y)$ avec probabilité $1-\rho(y)$. Comme il existe $m\in \mathbb{N}$ tel que $X=\bigcup_{i=1}^mT^i(F)$, et que $\rho$ vaut 1 sur F, quelque soit son point de départ, la marche aléatoire s'arrête après au plus m étapes.

Soit f la fonction qui à $x\in X$ associe le nombre d'étapes espéré avant la fin de la marche. Alors $f:X\longrightarrow \mathbb{R}$ est une fonction continue. De plus, si $x\notin U$, alors $\rho(x)=0$ donc la marche partant de x se déplace à l'instant suivant vers $T^{-1}x$. Donc $f(T^{-1}x)=f(x)-1$, c'est-à-dire $T^{-1}x\notin E_f$. On en déduit que $E_f\subset T(U)$. Comme $\forall i\in \{1,2,...,n-1\}, U\cap T^iU=\emptyset$, on en déduit que $\forall i\in \{ 1,2,...,n-1\}, E_f\cap T^i(E_f)=\emptyset$.

Finalement, (X,T) a la propriété topologique forte de Rokhlin.
\end{proof}

\begin{propo}
Si (X,T) est un système minimal, il a la propriété de marqueur.
\end{propo}

\begin{proof}
Soit $n\in \mathbb{N}$. Soit $x\in X$. $x,T(x),..., T^{n-1}(x)$ sont distincts. Donc il existe une boule ouverte B centrée en x vérifiant :
\[\forall i\in \{ 1,2,...,n-1\} , B\cap T^i(B)=\emptyset\]
L'orbite de x est dense dans X donc $\bigcup_{i=1}^\infty T^i(B)=X$. Par compacité de X, il existe $m\in \mathbb{N}$ tel que $\bigcup_{i=1}^m T^i(B)=X$. B est donc un n-marqueur ouvert de (X,T).

\end{proof}

\begin{Lemme}\label{bord}
Supposons que X est de dimension finie d, et (X,T) apériodique. Soit U un ouvert de X. Alors pour tout $k\in \mathbb{N}$ et pour tout ouvert $V$ tel que $\partial U \subset V$, il existe un ouvert $U'$ tel que $U\subset U'\subset U\cup V$, $\partial U'\subset \partial U\cup V$ et et toute sous-famille de $\{ \partial U', T(\partial U'),..., T^{k-1}(\partial U')\}$ de cardinal >d est d'intersection vide.
\end{Lemme}

\begin{Lemme}\label{tours}
Supposons que X est de dimension finie d et (X,T) est un système apériodique. Soit $N\in \mathbb{N}$. Soient U et V deux ouverts de X. Alors si on pose $m=(2d+2)N-1$, et si :
\begin{enumerate}
\item $\forall i\in \{ 1,2,...,N-1\} ,\bar{U}\cap T^i\bar{U}=\emptyset $
\item $\forall i\in \{ 1,2,...,m\} ,\bar{V}\cap T^i\bar{V}=\emptyset $
\end{enumerate}
alors il existe un ouvert $W\subset X$ tel que :
\begin{enumerate}
\item $\bar{U}\subset \bar{W}$
\item $\bar{V}\subset \bigcup_{i=1}^m T^i(\bar{W})$
\item $\forall i\in \{ 1,2,...,N-1\} ,\bar{W}\cap T^i\bar{W}=\emptyset $
\end{enumerate}
\end{Lemme}

Soit un réel $\rho >0$ et $A\subset X$. On notera :
\[ B_\rho (A)=\{ x\in X | d(x,A)\leq \rho \} \]

\begin{proof}
On pose : $R=\bar{V}\backslash \bigcup_{i=1}^mT^iU$ et on va construire l'espace W sous la forme suivante :
\[ W=U\cup \bigcup_{H\in \Phi}T^{-c(H)N}H\]
où $\Phi$ est un recouvrement ouvert fini de R, et pour tout H dans $\Phi$, c(H) est un entier tel que $1\leq c(H)N<m$.

Alors on aura : $\bar{U}\subset \bar{W}$ et $\bar{V}\subset \bigcup_{i=1}^m T^i(\bar{W})$ . En effet, soit $x\in V$. Si $x\in \bigcup_{i=1}^mT^iU$ alors comme $\bar{U}\subset \bar{W}$, $x\in \bigcup_{i=1}^m T^i(\bar{W})$. Sinon, $x\in R$. Alors il existe $H$ dans $\Phi$ tel que $x\in H$ : d'où $x\in T^{c(H)N}\bar{W}$. Comme $1\leq c(H)N<m$, on a bien $x\in \bigcup_{i=1}^m T^i(\bar{W})$.\\

On va de plus construire $\Phi$ et $c(H)$ pour tout $H\in \Phi$, de manière à ce que :
\begin{enumerate}
\item en posant $O=\bigcup_{H\in \Phi}\bar{H}$, $\forall i\in \{ 1,...,m\} , O\cap T^i(O)=\emptyset$
\item c est une fonction $c:\Phi \to \{1,...,2d+1\}$ telle que :
\[\forall H\in \Phi, \forall i\in \{ -(N-1),...,N-1\} , T^i\bar{H}\cap T^{c(H)N}\bar{U} =\emptyset \] \\
\end{enumerate}

Vérifions qu'alors on aura bien : $\forall i\in \{ 1,2,...,N-1\} ,\bar{W}\cap T^i\bar{W}=\emptyset $. Supposons le contraire. Il existe x et y dans $\bar{W}$ et $i\in \{ 1,2...,N-1\}$ tels que $x=T^iy$. On distingue cinq cas :

\begin{enumerate}
\item si x et y sont dans $\bar{U}$, cela contredit la propriété : $\forall i\in \{ 1,2,...,N-1\} ,\bar{U}\cap T^i\bar{U}=\emptyset $
\item s'il existe $H\in \Phi$ tel que x et y sont dans $T^{-c(H)N}\bar{H}$ alors ils sont dans  $T^{-c(H)N}O$. Comme $i<m$, cela contredit la propriété énoncée sur $O$.
\item s'il existe $H_1$ et $H_2$ distincts dans $\Phi$ tels que $x\
 \in T^{-c(H_1)N}\bar{H_1}$ et $y\in T^{-c(H_2)N}\bar{H_2}$, alors $T^{i+c(H_1)N}(y)\in H_1$ et $T^{c(H_2)N}(y)\in H_2$. Donc ces deux éléments sont dans $O$. Or $0<|i+c(H_1)N-c(H_2)N|\leq (2d+1)N<m$. Donc cela contredit la propriété sur $O$.
\item s'il existe $H\in \Phi$ tel que x est dans $T^{-c(H)N}\bar{H}$ et $y\in \bar{U}$, alors $T^{c(H)N}(x)\in \bar{H}\cap T^{i+c(H)N}(\bar{U})$. Donc $T^i\bar{H}\cap T^{c(H)N}\bar{U} \neq \emptyset$. Cela contredit la propriété énoncée sur c.
\item s'il existe $H\in \Phi$ tel que y est dans $T^{-c(H)N}\bar{H}$ et $x\in \bar{U}$, on conclut comme précédemment.\\
\end{enumerate}

Construisons maintenant $\Phi$ et $c$.\\

V vérifie : $\forall i\in \{ 1,2,...,m\} ,\bar{V}\cap T^i\bar{V}=\emptyset $. Donc R aussi, et on peut donc trouver un réel $\rho >0$ tel que $\forall i\in \{ 1,2,...,m\}, B_\rho(R)\cap T^i(B_\rho(R))=\emptyset $.

D'après \ref{bord}, quitte à agrandir U tout en gardant la condition : $\forall i\in \{ 1,2,...,N-1\} ,\bar{U}\cap T^i\bar{U}=\emptyset $, on peut supposer que toute sous-famille de $\{ \partial U, T^1(\partial U), ..., T^m(\partial U)\}$ de cardinal >d est d'intersection vide.

Alors on peut choisir $0<\delta <\rho$ tel que :
\[ |i\in \{1,...,m\} , T^i(\bar{U})\cap B_\delta(x) \neq \emptyset \} |\leq d  \]
En effet, supposons le contraire. Alors on peut trouver des suites $(x_n)_{n\in \mathbb{N}}\in R^\mathbb{N}$ et $(\delta_n)_{n\in \mathbb{N}}\in (\mathbb{R}^{+*})^\mathbb{N}$ telles que $ \delta_n \to_{n\to \infty} 0$ et :
\[ \forall n\in \mathbb{N}, |i\in \{1,...,m\} , T^i(\bar{U})\cap B_\delta(x_n) \neq \emptyset \} |> d \]
R est fermé dans le compact X donc est compact. On peut donc supposer (quitte à extraire) que la suite $(x_n)_{n\in \mathbb{N}}$ a une limite $x\in R$.
De plus, on peut trouver d+1 indices dans $\{ 1,2,...,m\}$, notés $i_1,i_2,...,i_{d+1}$, qui appartiennent chacun à une infinité d'ensembles $\{ i\in \{1,...,m\} , T^i(\bar{U})\cap B_\delta(x_n) \neq \emptyset \}$. Alors, quitte à extraire encore de la suite $(x_n)_{n\in \mathbb{N}}$, on peut supposer qu'il existe, pour tout $n\in \mathbb{N}$ et $1\leq l\leq d+1$, un élément $y_n^l \in T^{i_l}(\bar{U})$ tel que $d(y_n^l,x_n)\leq \delta_n$.
Alors les suites $(y_n^l)_{n\in \mathbb{N}}$ convergent toutes vers $x$. Donc $x\in R\cap \bigcap_{l=1}^{d+1}T^{i_l}\bar{U}$. Donc $\bigcap_{l=1}^{d+1}T^{i_l}(\partial \bar{U})\neq \emptyset$. C'est une contradiction, donc on peut bien choisir un tel $\delta$.

Par compacité de R, on peut trouver un recouvrement ouvert fini de R, $\Phi$, par des boules fermées de rayon $\delta$. Alors :
\[ \forall H \in \Phi, |\{ i\in \{1,...,m\} , T^i(\bar{U})\cap \bar{H} \neq \emptyset \} |\leq d  \]
En posant $O=\bigcup_{H\in \Phi}\bar{H}$, on a bien : $\forall i\in \{ 1,...,m\} , O\cap T^i(O)=\emptyset$

Enfin, pour $H\in \Phi$, on peut trouver un élément $c(H)$ tel que :
\[ \forall i\in \{ -(N-1),...,N-1\} , T^i\bar{H}\cap T^{c(H)N}\bar{U} =\emptyset \]
En effet, pour $1\leq k \leq 2d+1$, notons : $I_k=\{ kN-(N-1),...,kN+N-1\}$. Or, pour tout k, $I_k \subset \{ 1,2,...,m\}$ et d'autre part, $I_k$ et $I_{k+2}$ sont disjoints. Donc si $i\in \{ i\in \{1,...,m\} , T^i(\bar{U})\cap \bar{H} \neq \emptyset \}$, i est dans au plus deux de ces ensembles $I_k$. Or il y a au plus d tels indices i. Donc il existe un ensemble $I_{k_0}$ qui ne contient aucun indice i de ce type. En posant $c(H)=k_0$, on obtient la propriété désirée.
\end{proof}

\begin{thm}\label{marqapé}
Supposons que (X,T) est un système apériodique de dimension finie. Alors (X,T) a la propriété de marqueur.
\end{thm}

\begin{proof}
Soit $n\in \mathbb{N}$. Posons comme dans \ref{tours}, on pose $m=(2d+2)N-1$. Comme (X,T) est apériodique et X compact, on peut recouvrir X par un nombre fini d'ouverts $U_1, U_2, ...,U_s$ tels que :
\[ \forall j\in \{1,2,...,s\}, \forall i\in \{ 1,2, ...,m\}, \bar{U_j} \cap T^i(\bar{U_j})=\emptyset \]
On va construire un n-marqueur par récurrence. On pose $W_1=U_1$.
Soit $k\leq s$. Supposons qu'il existe un ouvert $W_k$ tel que $\bigcup_{i=1}^k\bar{U}_i \subset \bigcup_{i=0}^kT^i(\bar{W_k})$ et $\forall i\in \{1,2,...,N-1\}, \bar{W_k}\cap T^i\bar{W_k}=\emptyset$. En appliquant \ref{tours} à $W_k$ et $U_{k+1}$, on trouve un ouvert $W_{k+1}$ tel que $\bar{W}_k \subset \bar{W}_{k+1}$, $\bar{U}_{k+1} \subset \bigcup_{i=1}^kT^i(\bar{W}_{k+1})$ et $\forall i\in \{1,2,...,N-1\}, \bar{W}_{k+1}\cap T^i\bar{W}_{k+1}=\emptyset$. On a alors $\bigcup_{i=1}^{k+1}\bar{U}_i \subset \bigcup_{i=0}^mT^i(\bar{W}_{k+1})$.
Finalement, en posant $W=W_s$, on obtient un n-marqueur W.
\end{proof}

\subsubsection{Application aux plongements}

\begin{Lemme}\label{approx1}
Soit $S\in \mathbb{N}$. Soit $\alpha$ un recouvrement ouvert fini de X tel que $ord(\beta)<Sd$.
Pour tout $U\in \alpha$, on se donne un point $q_U\in U$ et un N-uplet $v_U \in ([0,1]^d)^N$. Soit $\epsilon >0$.
Alors on peut trouver une fonction $F:X\to ([0,1]^d)^N$ telle que :
\begin{enumerate}
\item $\forall U\in \alpha ,\Arrowvert F(q_U)-v_U \Arrowvert _\infty <\epsilon$
\item $\forall x\in X, F(x) \in Conv(\{ F(q_U) | x\in U \}$
\item s'il existe $0\leq l <N-4S$, $\lambda \in ]0,1]$, et $x,y,x',y'  \in X$ tels que :
\[ \lambda F(x)|_l^{l+4S-1} + (1-\lambda)F(y)|_{l+1}^{l+4S}= \lambda F(x')|_l^{l+4S-1} + (1-\lambda)F(y')|_{l+1}^{l+4S} \]
alors il existe $U\in \beta$ tel que x et x' sont dans $U$.
\end{enumerate}
\end{Lemme}

\begin{proof}

On définit la fonction F ainsi :
\[ \forall x\in X, F(x) =\sum_{U\in \alpha}\rho_U(x)\vec{v_U} \]
où $\{ \rho_U\}_{U\in \alpha}$ est une partition de l'unité subordonnée à $\alpha$ telle que pour tout $U\in \alpha$, $\rho_U(q_U)=1$.
On définira $\{ \vec{v_U}\}_{U\in \alpha}$ par la suite. La propriété 2. sera alors immédiatement vérifiée.

Soit $x\in X$. On pose : $\alpha_x =\{ U\in \alpha \: | \: \rho_U(x)>0\}$. La propriété 3. s'écrit alors :

\[ \lambda \sum_{U\in \alpha_x}\rho_U(x)\vec{v_U}|_l^{l+4S-1} + (1-\lambda)\sum_{U\in \alpha_y}\rho_U(y)\vec{v_U}|_{l+1}^{l+4S}= \lambda \sum_{U\in \alpha_{x'}}\rho_U(x')\vec{v_U}|_l^{l+4S-1} + (1-\lambda)\sum_{U\in \alpha_{y'}}\rho_U(y')\vec{v_U}|_{l+1}^{l+4S} \]

Considérons la matrice M dont les colonnes sont les vecteurs de $\{ \vec{v_U}|_l^{l+4S-1}\}_{U\in \alpha_x \cup \alpha_{x'}} \cup \{ \vec{v_U}|_{l+1}^{l+4S}\}_{U\in \alpha_y\cup \alpha_{y'}}$. Alors comme $ord(\alpha)<Sd$, M a 4S lignes et k colonnes, où $k\leq 4Sd$. On peut alors appliquer \ref{matrice} de la section 3.6. à la matrice carrée obtenue à partir de M en ne gardant que les k premières lignes. On en déduit que pour presque tout choix de $\{ \vec{v_U}\}_{U\in \alpha}$ , les vecteurs de $\{ \vec{v_U}|_l^{l+4S-1}\}_{U\in \alpha_x \cup \alpha_{x'}} \cup \{ \vec{v_U}|_{l+1}^{l+4S}\}_{U\in \alpha_y\cup \alpha_{y'}}$ sont linéairement indépendants, donc la propriété 3. est vérifiée.

Comme précédemment, comme il existe un nombre fini de familles de la forme $\alpha_1^x$, on peut choisir la famille $\{ \vec{v_U}\}_{U\in \alpha_1^x}$ dans un ensemble de complémentaire négligeable, donc de façon à vérifier aussi la propriété 1.

\end{proof}

\begin{thm}
Supposons que (X,T) est une extension d'un système dynamique (Z,S) apériodique de dimension finie, et que $mdim(X,T)<\frac{d}{16}$. Alors (X,T) se plonge dans le décalage $(([0,1]^{d+1})^\mathbb{Z},\sigma)$.
\end{thm}

\begin{proof}

Soit $\pi:(X,T) \to (Z,S)$ l'extension de (X,T) à (Z,S).
Nous allons montrer qu'il existe une fonction $g\in C(X,[0,1]^d)$ telle que $I_g\times \pi$ est un plongement de (X,T) dans $(([0,1]^d)^\mathbb{Z},\sigma)\times (Z,S)$. Alors, d'après le théorème de Jaworski, il existe d'autre part un plongement $I_h$ de (Z,S) dans $([0,1]^\mathbb{Z},\sigma)$, où $h\in C(Z,[0,1])$. Donc on obtiendra un plongement $I_f$ de $(X,T)$ dans $(([0,1]^{d+1})^\mathbb{Z}, \sigma)$, en posant $f=(g,h\circ \pi)$.

On pose :
\[ D_\epsilon = \{ f\in C(X,[0,1]^d) | I_f\times \pi \: est\: \epsilon-injective \} \]

Montrons que pour tout $\epsilon$, $D_\epsilon$ est dense dans $C(X,[0,1]^d)$. Soit $\tilde{f}\in C(X,[0,1]^d)$ et $\delta<0$. Posons :

\[ m_{dim}=
\left\{
\begin{array}{ll}
\frac{1}{32} &\text{si }mdim(X,T)=0 \\
mdim(X,T)&\text{sinon}
\end{array}
\right.
\]
On a $0<m_{dim}<\frac{d}{16}$.

Soit $\alpha$ un recouvrement ouvert fini de X tel que $ \max_{U\in \alpha}diam(f(U))<\frac{\delta}{2}$, et $\max_{U\in \alpha}diam(U)<\epsilon$. Soit $\epsilon '>0$ tel que $16m_{dim}(1+2\epsilon ')<d$. Alors on peut trouver $N\in \mathbb{N}$, divisible par 16, tel que $\frac{1}{N} D(\alpha_0^{N-1})<(1+\epsilon ')m_{dim}$. Soit $\gamma \succ \alpha_0^{N-1}$ un recouvrement ouvert fini de X tel que $ord(\gamma) =D(\alpha_0^{N-1})$. Posons $M=\frac{N}{2}$ et $S=\frac{N}{16}$. Alors : $ord(\gamma)<Sd$.

On peut appliquer \ref{approx1} à $\gamma$ et S. On choisit, pour tout $U\in \gamma$, un point $q_U \in U$ et on définit le N-uplet $v_U=(\tilde{f}(T^iq_U))_{i=0}^{N-1}$ On obtient une fonction $F:X\to ([0,1]^d)^N$.

D'autre part, d'après \ref{marqapé}, (Z,S) a la propriété de marqueur. Donc d'après \ref{Rok}, (Z,S) a la propriété de Rokhlin topologique forte. Donc il existe une fonction continue $n:X\to \mathbb{R}$ telle qu'en définissant :
\[ E_n=\{ x\in S \: | \: n(T(x))\neq n(x)+1\} \]
on a :
\[ \forall i\in \{ 1,2,...,N-1\}, E_n\cap T^i(E_n)=\emptyset \]

Posons alors, pour tout x dans X : $\bar{n}(x)=\lceil n(\pi (x)) \rceil$ mod M,  $\underline{n}(x)=\lfloor n(\pi (x)) \rfloor$ mod M et $n'(x)=\{ n(x)\}$. On définit, pour tout x dans X :
\[f(x)= (1-n'(x))F(T^{-\underline{n}(x)}x)|_{\underline{n}(x)}+n'(x)F(T^{-\bar{n}(x)}x)|_{\bar{n}(x)}\]
On obtient une fonction $f:X\to [0,1]^d$.

Montrons que f est continue. Soit $x\in X$. Si $n(x) \notin \mathbb{Z}$, alors il existe un voisinage de x sur lequel $\bar{n}$ et $\underline{n}$ sont constantes, et $n'$ continue. Donc f est continue en x. Si en revanche $x\in \mathbb{Z}$, fixons $\epsilon >0$. Alors pour $x'$ suffisamment proche de x, on obtient l'un ou l'autre des cas :
\begin{enumerate}
\item $n(x)\leq n(x')<n(x)+\frac{\epsilon}{2}$ (alors $\underline {n}(x')=n(x)$ et $n'(x')<\frac{\epsilon}{2}$)
\item $n(x)-\frac{\epsilon}{2}<n(x')<n(x)$ (alors $\bar{n}(x')=\bar{n}(x)=n(x)$ et $n'(x')>1-\frac{\epsilon}{2}$)
\end{enumerate}

Dans le premier cas :
\[ \Arrowvert f(x)-f(x')  \Arrowvert   \leq \Arrowvert F(T^{-n(x)}(x'))|_{n(x)}- F(T^{-n(x)}(x))|_{n(x)} \Arrowvert + \Arrowvert n'(x')(F(T^{-\bar{n}(x')}(x'))|_{n(x)+1}-F(T^{-\bar{n}(x)}(x))|_{n(x)})  \Arrowvert  \]
On a :
\[  \Arrowvert (F(T^{-\bar{n}(x')}(x'))|_{n(x)+1}-F(T^{-\bar{n}(x)}(x))|_{n(x)})  \Arrowvert \leq 2\]
Donc pour x' suffisamment proche de x, on obtient :
\[ \Arrowvert f(x)-f(x')  \Arrowvert   \leq \epsilon (1+d) \]
On obtient la même chose dans le second cas. On en déduit que f est continue.\\

Montrons que  $\Arrowvert \tilde{f}-f \Arrowvert_\infty < \epsilon$. Soit $x\in X$. Alors :
\[ f(x)\in Conv(\{ F(q_U)|_{\underline{n}(x)} \: |\:  T^{-\underline{n}(x)}x\in U\} \cup \{ F(q_U)|_{\bar{n}(x)} \: |\: T^{-\bar{n}(x)}x\in U\}) \]
Or, pour tous $0\leq n <N$ et $U\in \beta$ tels que $T^{-n}(x)\in U$, comme $F(q_U)=v_U $, on a :
\[ \Arrowvert F(q_U)|_n-\tilde{f}(x)\Arrowvert = \Arrowvert v_U|_n-\tilde{f}(x)\Arrowvert = \Arrowvert \tilde{f}(T^nq_U)-\tilde{f}(x)\Arrowvert \]
Comme x et $T^n(q_U)$ sont tous les deux dans $T^n(U)$, on a :
\[ \Arrowvert F(q_U)|_n-\tilde{f}(x)\Arrowvert \leq diam (\tilde{f}(T^n(U))) \]
Comme $\beta \succ \alpha_0^{M-1}$ et $n<N$, on peut trouver $V\in \alpha$ tel que $U \subset T^{-n}(V)$. Donc :
\[ \Arrowvert F(q_U)|_n-\tilde{f}(x)\Arrowvert \leq \max_{V\in \alpha} diam (\tilde{f}(V))<\epsilon \]
Finalement, $\Arrowvert f(x)-\tilde{f}(x) \Arrowvert <\epsilon$.\\

Enfin, montrons que $f\in D_\epsilon$. Soient x et y dans X, tels que $(I_f\times \pi)(x)=(I_f\times \pi)(y)$. Alors pour tout $a\in \mathbb{Z}$, $f(T^ax)=f(T^ay)$. De plus, comme $\pi \circ T=S\circ \pi$ et $\pi (x)=\pi (y)$, on a pour tout $a\in \mathbb{Z}$, $\pi (T^ax)=\pi(T^ay)$. Donc :
\[\forall a\in \mathbb{Z},\: n(\pi(T^ax))=n(\pi(T^ay)) \]

Comme $\frac{M}{2}-2+\frac{3M}{2}=2M-2<N$, il existe au plus un indice $j\in \{ -\frac{3}{2}M,..., \frac{M}{2}-2$ tel que $n(T^{j+1}(x))\neq n(T^j(x))+1$. Alors d'après \ref{n1}, on peut trouver un indice $r \in \{ \frac{3M}{2},..., 0\}$. tel que $\underline{n}(T^rx) \leq \frac{M}{2}$ et pour tout $s\in \{ r,..., r+\frac{M}{2}-1\}$, $\underline{n}(T^sx)=\underline{n}(T^rx)+s-r$ et $\bar{n}(T^sx)=\bar{n}(T^rx)+s-r$.

Alors, en posant $\lambda=n'(T^rx)$ et $a=\underline{n}(T^rx)$, l'égalité :
\[ \forall s\in \{ r,..., r+\frac{M}{2}-1\}, f(T^sx)=f(T^sy)  \]
 s'écrit :
\[  \lambda F(T^{r-a}x)|_a^{a+4S-1} + (1-\lambda)F(T^{r-a-1}x)|_{a+1}^{a+4S}= \lambda F(T^{r-a}y)|_a^{a+4S-1} + (1-\lambda)F(T^{r-a-1}y)|_{a+1}^{a+4S} \]

Il existe donc un ouvert $U\in \gamma$ tel que $T^{r-a}x$ et $T^{r-a}y$ sont dans U. Comme $\gamma \succ \alpha$ et $-N\leq r-a\leq 0$, il existe $V\in \alpha$ tel que $T^{r-a}x$ et $T^{r-a}y$ sont dans $T^{r-a}(V)$ et $T^{r-a}(V)$, d'où x et y sont dans V. On a supposé que $\max_{U\in \alpha}diam(U)<\epsilon$, donc $d(x,y)<\epsilon$.

Comme ceci est vrai pour tout couple (x,y), on conclut $f\in D_\epsilon$.

\end{proof}

\subsection{Quelques lemmes techniques}

On rassemble ici quelques lemmes utilisés dans les sections précédentes pour établir les résultats de densité.

\begin{Lemme}\label{indep}
Soit $V=Vect(v_1,v_2,...,v_r)$ un sous-espace vectoriel de $\mathbb{R}^m$ de dimension r>0. Soit $s\in \mathbb{N}$.Alors :
\begin{enumerate}
\item Si $r+s\leq m$ alors pour presque tout $(v_{r+1},v_{r+2},...,v_{r+s})\in ([0,1]^m)^s$ (relativement à la mesure de Lebesgue sur $([0,1]^m)^s$), $Vect(v_1,v_2,...,v_{r+s})$ est de dimension r+s.
\item Si $r+s\leq m+1$, pour presque tout $(v_{r+1},v_{r+2},...,v_{r+s})\in ([0,1]^m)^s$, $Vect(v_2-v_1,v_3-v_1,...,v_{r+s}-v_1)$ est de dimension r+s-1.
\end{enumerate}
\end{Lemme}

\begin{proof}
Supposons $r+s\leq m$. On procède par récurrence sur s. Pour s=0, il n'y a rien à démontrer. Soit $i\in \mathbb{N}$ tel que $r+i< m$ et pour presque tout $(v_{r+1},v_{r+2},...,v_{r+i})\in ([0,1]^m)^i$ , $Vect(v_1,v_2,...,v_{r+i})$ (qu'on note $V_i$), est de dimension r+i. Alors, pour tout vecteur $v_{r+i+1}$ qui n'appartient pas à $V_i\cap [0,1]^m$, $Vect(v_1,v_2,...,v_{r+i},v_{r+i+1})$ est de dimension r+i+1. Si on note $\lambda$ la mesure de Lebesgue sur $[0,1]^m$, on a $\lambda(V_i\cap [0,1]^m)=0$ car $r+i<m$. Donc pour presque tout $(v_{r+1},v_{r+2},...,v_{r+i+1})\in ([0,1]^m)^{i+1}$ , $Vect(v_1,v_2,...,v_{r+i+1})$ est de dimension r+i+1.

Finalement, pour presque tout $(v_{r+1},v_{r+2},...,v_{r+s})\in ([0,1]^m)^s$, $Vect(v_1,v_2,...,v_{r+s})$ est de dimension r+s. De plus, on en déduit que $Vect(v_2-v_1,v_3-v_1,...,v_{r+s}-v_1)$ est de dimension r+s, ce qui règle une partie du cas 2.

En utilisant ce résultat pour r+s=m, on obtient que pour presque tout $(v_{r+1},v_{r+2},...,v_{m})\in ([0,1]^m)^{m-r}$, $Vect(v_1,v_2,...,v_{m})$ est de dimension m, d'où $Vect(v_2-v_1,v_3-v_1,...,v_{m}-v_1)$ (qu'on note $\tilde{V}$) est de dimension $m-1$. On en déduit que $\lambda(\tilde{V}\cap [0,1]^m)=0$. Donc de même, pour presque tout vecteur $v_{m+1}$ dans $[0,1]^m$, $Vect(v_2-v_1,v_3-v_1,...,(v_{m+1}+v_1)-v_1)$ est de dimension m, ce qui permet de conclure.
\end{proof}

\begin{Lemme}\label{matrice}
Soit $r\in \mathbb{N}$ et M une matrice carrée de taille n à coefficients dans $\{1,2,...,r\}$, telle qu'il n'y a pas de coefficients égaux sur une même ligne ou une même colonne. Alors pour presque tous $t_1,t_2,...,t_r$ dans $\mathbb{R}$, la matrice $A(t_1,t_2,...,t_r)$ définie par :
\[ A(t_1,t_2,...,t_r)_{i,j}=t_{M_{i,j}} \]
est inversible.
\end{Lemme}

\begin{proof}
On va montrer que le polynôme à r variables $det(A(t_1,t_2,...,t_r))$ est non nul.

Supposons que 1 apparaît s fois dans M (on peut supposer sans perte de généralité que s est non nul). Alors en écrivant :

\[ det(A)= \sum_{\sigma \in \Sigma_n}\epsilon(\sigma)a_{1\sigma(1)}a_{2\sigma(2)}...a_{n\sigma(n)}  \]

on remarque que $det(A(t_1,t_2,...,t_r))$ est un polynôme en la variable $t_1$, dont le coefficient dominant est lui-même un polynôme en les variables $t_2,...,t_r$. Plus précisément, si s=n, ce coefficient dominant est égal à 1, sinon c'est le déterminant de la matrice extraite de $A(t_1,t_2,...,t_r)$ en supprimant les lignes et les colonnes où $t_1$ apparait. Comme cette matrice extraite est de taille <n, on en conclut qu'on pourra procéder par récurrence : si son coefficient est un polynôme non nul, alors le polynôme $det(A(t_1,t_2,...,t_r))$ sera non nul.
\end{proof}

\begin{Lemme}\label{matrice2}
Soit $r\in \mathbb{N}$, $k\in \mathbb{N}$. Soit M une matrice carrée de taille $(2k-1)\times 2k$ à coefficients dans $\{1,2,...,r\}$, telle qu'il n'y a pas de coefficients égaux sur une même ligne ou une même colonne, et telle que chaque coefficient apparait au plus deux fois dans M. Alors pour presque tous $t_1,t_2,...,t_r$ dans $\mathbb{R}$, les colonnes de la matrice $A(t_1,t_2,...,t_r)$ définie par :
\[ A(t_1,t_2,...,t_r)_{i,j}=t_{M_{i,j}} \]
sont des vecteurs affinement indépendants.
\end{Lemme}

\begin{proof}
Procédons par récurrence sur k. Supposons qu'il existe un entier k tel que le résultat est vrai pour tout matrice de taille $(2k-3)\times (2k-2)$, et montrons le résultat pour tout matrice M de taille $(2k-1)\times 2k$.  Si chaque coefficient de M n'apparait qu'une fois, alors \ref{indep} (cas 2) permet de conclure. Sinon, il existe un entier i qui apparait deux fois. Soit j tel que i n'est pas sur la j-ième colonne. Alors soit N la matrice obtenue à partir de M en retranchant la j-ème colonne à toutes les autres, et en la supprimant. N est carrée de taille $2k\times 2k$. Pour conclure, il suffit de montrer que $det(A_N(t_1,t_2,...,t_r))\neq 0$ pour presque tout choix de $t_1,t_2,...,t_r$.

Comme dans le lemme précédent, on considère $det(det(A_N(t_1,t_2,...,t_r))$ comme un polynôme en la variable $t_i$. Alors le coefficient en $t_i^2$ est le déterminant de la matrice $A_{\tilde{N}}(t_1,t_2,...,t_r)$, où $\tilde{N}$ est la matrice obtenue à partir de N en supprimant les lignes et les colonnes contenant i.

Soit $\tilde{M}$ la matrice obtenue à partir de M en supprimant les lignes et les colonnes contenant i. Alors M est une matrice de taille $(2k-3)\times (2k-2)$ qui vérifie les propriétés de l'énoncé. Par hypothèse de récurrence, pour presque tous $t_1,.., t_{i-1}, t_{i+1},..,t_r$ dans $\mathbb{R}$, les colonnes de $A_{\tilde{M}}(t_1,t_2,...,t_r)$ sont des vecteurs affinement indépendants.

De plus, on remarque que $\tilde{N}$ est la matrice obtenue à partir de $\tilde{M}$ en en retranchant la j-ème colonne à toutes les autres, et en la supprimant. On en déduit que $det(A_{\tilde{N}}(t_1,t_2,...,t_r))\neq 0$ pour presque tout choix de $t_1,t_2,...,t_r$. Donc le polynôme en la variable $t_i$, $det(det(A_N(t_1,t_2,...,t_r))$ est non nul. Finalement, $det(A_N(t_1,t_2,...,t_r))\neq 0$ pour presque tout choix de $t_1,t_2,...,t_r$.
\end{proof}

\begin{Lemme}\label{n1}
Soit une fonction $n:X\to \mathbb{R}$, et $M\in \mathbb{N}$ un entier pair. Supposons qu'il existe au plus un indice $j\in \{ -\frac{3}{2}M,..., \frac{M}{2}-2\}$ tel que $n(T^{j+1}(x))\neq n(T^j(x))+1$. Alors on peut trouver un indice $r\in \{ \frac{3}{2}M,...,0\} $ tel que :
\begin{enumerate}
\item $\lfloor n(T^rx) \rfloor$ mod M $\leq \frac{M}{2}$
\item $\forall s\in \{ r,..., r+\frac{M}{2}-1\},  \lfloor n(T^sx) \rfloor$ mod M $=\lfloor n(T^rx) \rfloor$ mod M+s-r
\item$\forall s\in \{ r,..., r+\frac{M}{2}-1\},  \lceil n(T^sx) \rceil$ mod M $=\lceil n(T^rx) \rceil$ mod M+s-r
\end{enumerate}
\end{Lemme}

\begin{proof}
L'ensemble $\{ -\frac{3}{2}M,..., \frac{M}{2}-2\}$ étant de longueur M+(M-1), l'un des ensembles : $\{ -\frac{3}{2}M,..., j\} , \{ -\frac{3}{2}M,..., \frac{M}{2}-2\}$ est de longueur au moins égale à M. On note $\{ a,...,b\}$ cet ensemble : on a donc $ b-a \geq M-1$. Par définition de l'indice j, $\{ \lfloor n(T^ax) \rfloor,\lfloor n(T^{a+1}x) \rfloor,...,\lfloor n(T^bx) \rfloor\}$ est une famille d'au moins M entiers distincts. On peut donc trouver $a\leq k\leq b$ tel que $\lfloor n(T^kx) \rfloor$ mod M $ =0$.

Supposons que $b-k+1 \geq \frac{M}{2}$. Alors posons $r=k$. On a $\lfloor n(T^rx) \rfloor$ mod M $=0$. Soit $s\in \{ r,..., r+\frac{M}{2}-1\}$. Alors comme $a\leq s\leq b$, $n(T^sx)=n(T^rx) +s_r$, donc $\lfloor n(T^sx) \rfloor = \lfloor n(T^rx) \rfloor +s-r$. De plus, $0\leq r-s < \frac{M}{2}$, donc $\lfloor n(T^sx) \rfloor$ mod M $ = \lfloor n(T^rx) \rfloor$ mod M $ +s-r$. De même, $\lceil n(T^sx) \rceil$ mod M $ = \lceil n(T^rx) \rceil$ mod M $ +s-r$.

Si au contraire $b-k+1<\frac{M}{2}$, alors on pose $r=k-\frac{M}{2}-1$. Alors a $0\geq r\geq a$. Comme $a\leq r\leq b$, on a $\lfloor n(T^rx) \rfloor = \lfloor n(T^kx) \rfloor +r-k = -\frac{M}{2}-1$. Donc on a bien $\lfloor n(T^rx) \rfloor$ mod M $\leq \frac{M}{2}$. Soit $s\in \{ r,..., r+\frac{M}{2}-1\}$. On a encore $a\leq s\leq b$, donc on peut conclure de la même façon.
\end{proof}

\bibliographystyle{alpha}
\bibliography{biblio}

\begin{thebibliography}{Jaw74}

\bibitem[CK05]{K}
Michel Coornaert and Fabrice Krieger.
\newblock Mean topological dimension for actions of discrete amenable groups.
\newblock {\em Discrete and continuous dynamical systems}, 13(3):779--793,
  August 2005.

\bibitem[Coo05]{Coo}
Michel Coornaert.
\newblock {\em Dimension topologique et systèmes dynamiques}, volume~14.
\newblock Société Mathématique de France, 2005.

\bibitem[Gut14]{YG}
Yonatan Gutman.
\newblock Mean dimension and jaworsky-type theorems.
\newblock Preprint. arXiv:1208.5248v5, 2014.

\bibitem[Jaw74]{Jaw}
Allan Jaworski.
\newblock {\em The Kakutani-Bebutov thorem for groups}.
\newblock PhD thesis, University of Maryland, 1974.

\bibitem[Lin99]{LIN}
Elon Lindenstrauss.
\newblock Mean dimension, small entropy factors and embedding theorem.
\newblock {\em Publications mathématiques de l'I.H.E.S.}, 89:227--262, 1999.

\bibitem[LW00]{LW}
Elon Lindenstrauss and Benjamin Weiss.
\newblock Mean topological dimension.
\newblock {\em Israel Journal of Mathematics}, (115):1--24, 2000.

\end{thebibliography}
\nocite{*}

\end{document}